\documentclass[final,1p,times,margin=1in]{elsarticle} 

\usepackage{amssymb}
 \usepackage{amsthm}
\usepackage{amscd}
\usepackage{amsmath}
\usepackage{amsfonts}
\usepackage{color}
\usepackage{graphicx}
\usepackage{tikz}
\usepackage{url}
\newtheorem{theorem}{Theorem}[section]

\newtheorem{corollary}[theorem]{Corollary}
\newtheorem{remark}[theorem]{Remark}
\newtheorem{example}[subsection]{Example}
\usepackage[ruled]{algorithm2e}
\usepackage{mathrsfs}
\usepackage{titletoc}
\usepackage{float}

\usepackage{xcolor}

\newcommand{\best}[1]{\textbf{\textcolor{red}{#1}}}
\newcommand{\second}[1]{\emph{\textcolor{orange}{#1}}}

\usepackage{caption}
\usepackage{subcaption}
\usepackage{booktabs}
\usepackage{geometry}

\journal{******}

\begin{document}

\begin{frontmatter}

\title{An Efficient Entropy Flow on Weighted Graphs: Theory and Applications}

\author[ruc]{Juan Zhao}
\ead{zhaojuan0509@ruc.edu.cn}

\author[ruc]{Jicheng Ma}
\ead{2019202433@ruc.edu.cn}

\author[ruc]{Yunyan Yang\corref{cor1}}
\ead{yunyanyang@ruc.edu.cn}

\author[bnu]{Liang Zhao}
\ead{liangzhao@bnu.edu.cn}

\cortext[cor1]{Corresponding author}

\address[ruc]{School of Mathematics, Renmin University of China, Beijing, 100872, China}
\address[bnu]{School of Mathematical Sciences, Key Laboratory of Mathematics and Complex Systems of MOE,\\
Beijing Normal University, Beijing, 100875, China}

\begin{abstract}
We propose a novel entropy flow on weighted graphs, which provides a principled framework that characterizes the evolution of probability distributions over graph structures while sharing geometric intuition with discrete Ricci flow. We provide its rigorous formulation, establish its fundamental theoretical properties, and prove the long-time existence and convergence of its solutions. To demonstrate its applicability, we employ entropy flow for community detection in real-world networks. Empirically, it achieves detection accuracy fully comparable to that of discrete Ricci flow. Crucially, by avoiding computations of optimal transport distances and shortest paths, our approach overcomes the fundamental computational bottleneck of Ollivier and Lin-Lu-Yau Ricci flows. As a result, entropy flow requires only $1.61\%$-$3.20\%$ of the computation time of Ricci flow. These results indicate that entropy flow provides a theoretically rigorous and computationally efficient framework for large-scale graph analysis.

\end{abstract}


\end{frontmatter}

\titlecontents{section}[0mm]
                       {\vspace{.2\baselineskip}}
                       {\thecontentslabel~\hspace{.5em}}
                        {}
                        {\dotfill\contentspage[{\makebox[0pt][r]{\thecontentspage}}]}
\titlecontents{subsection}[3mm]
                       {\vspace{.2\baselineskip}}
                       {\thecontentslabel~\hspace{.5em}}
                        {}
                       {\dotfill\contentspage[{\makebox[0pt][r]{\thecontentspage}}]}

\setcounter{tocdepth}{2}



\numberwithin{equation}{section}
\section{Introduction}

Developing mathematically well-founded dynamical models on graphs is a central problem in applied mathematics and network science, which arises naturally in a wide range of applications, including biological networks, social systems, and information processing architectures. Understanding how probability distributions, structural patterns, and functional organizations evolve on such discrete structures is fundamental for both theoretical analysis and practical modeling.

Entropy-based methods provide a natural framework for characterizing uncertainty, information flow, and complexity in networked systems. Originating from information theory and statistical mechanics, entropy and relative entropy have been extensively studied in the context of stochastic processes, diffusion dynamics, and random walks on graphs \cite{Shannon1948, Kullback1951, Thomas2005}. In network analysis, entropy-related quantities have been employed to quantify structural heterogeneity, robustness, and centrality \cite{Anand2009, Dehmer2008}. These approaches offer a probabilistic perspective for investigating the interplay between local interactions and global organization.

In parallel, geometric methods on graphs have attracted increasing attention over the past two decades. Inspired by Riemannian geometry, discrete analogues of Ricci curvature have been introduced to capture transport and connectivity properties of networks. Notable examples include Ollivier Ricci curvature \cite{Ollivier2009ricci} and Lin-Lu-Yau Ricci curvature \cite{Lin2011Ricci}. These notions provide a powerful geometric lens for analyzing network's structural stability and clustering behavior \cite{Sanhu2015graph, Ni2019community}. By linking curvature to optimal transport and random walk behavior, Ricci curvature establishes deep connections between geometry, probability, and graph theory.

Ricci flow is a evolution equation which plays a fundamental role in geometry and topology \cite{Hamilton1982three}, describing how a Riemannian metric on a  manifold  evolves over a parameter $t$. The equation
$$\partial_t g_{ij}=-2R_{ij},$$
where $g_{ij}$ is the metric tensor and $R_{ij}$ is the Ricci curvature tensor, deforms the metric to make curvature more uniform. Building upon notions of discrete curvature, several formulations of Ricci flow on graphs have been proposed and studied \cite{Ollivier2009ricci, Bai2024ollivier, Bai2025ricci, Ma2025modified, Li2026convergence}. Mathematically, discrete Ricci flow is formulated as a dynamical process that iteratively deforms edge weights based on local curvature. This offers a powerful framework that combines geometric insight with dynamical systems theory. Such methods have proven effective in revealing latent geometric structures and enhancing performance in tasks including community detection, anomaly identification, and network alignment \cite{Sanhu2015graph, Ni2018network, Ni2019community, Lai2022normalized, Ma2024evolution, Ma2025piecewise}.

However, despite their theoretical elegance and empirical success, existing Ricci flow methods face a significant computational bottleneck. The calculation of Ollivier or Lin-Lu-Yau curvature requires solving optimal transport problems and computing shortest paths, which becomes prohibitively expensive for large-scale networks and complicates the analysis of the resulting nonlinear dynamical system. From a theoretical perspective, the long-time existence of existing Ricci flow methods is not always guaranteed. To the best of our knowledge, only \cite{Ma2025modified} has established this property under general conditions. Furthermore, proving the convergence of the flow poses an even more non-trivial challenge \cite{Bai2024ollivier, Bai2025ricci}. On the practical side, these theoretical gaps make it difficult to ensure that the iterative process evolves as intended or to determine whether it will converge at all. Consequently, the high computational complexity of discrete Ricci flow, together with the lack of theoretical guarantees regarding its long-time existence and convergence, has severely constrained its broader adoption. This motivates the search for alternative dynamical frameworks on graphs that retain geometric intuition while being both computationally efficient and theoretically well-founded.

In this work, we propose a novel entropy flow on weighted graphs based on the Kullback–Leibler divergence. Our approach formulates the evolution of probability distributions on graph vertices as a nonlinear dynamical system driven by information-theoretic principles. The proposed entropy flow preserves essential geometric intuition analogous to that of discrete Ricci flow, while avoiding computations of optimal transport distances and shortest paths. This formulation leads to a substantial reduction in computational complexity and facilitates mathematical analysis.

From a theoretical perspective, we establish a rigorous foundation for the proposed entropy flow. We prove the uniqueness, long-time existence and convergence of solutions under mild assumptions. Several illustrative examples are provided to elucidate the geometric and structural interpretation of the flow. These results yield a well-posedness theory for an information-driven dynamical system on weighted graphs. From an application standpoint, we apply entropy flow to community detection in real-world networks. Experiments on multiple benchmark datasets demonstrate that the proposed method achieves accuracy comparable to that of Ricci-flow-based approaches in terms of ARI, NMI, and modularity, while requiring only a small fraction ($1.61\%$-$3.20\%$) of their computational time, highlighting the scalability of the proposed framework.

The main contributions of this paper are summarized as follows: (1) We introduce a novel entropy flow on weighted graphs and prove uniqueness, long-time existence and convergence for the proposed flow. (2) We develop efficient algorithms and validate the theoretical results through community detection experiments, while avoiding computations of optimal transport distances and shortest paths. These results indicate that entropy flow provides a mathematically rigorous and computationally efficient alternative to discrete Ricci flow.

\section{Definition of entropy flow and main results}

Let us first define an $\alpha$-lazy outward random walk centered at any vertex. Then we compute an entropy between two such random walks
centered at vertices of an edge. Finally,  we establish an evolution of each edge according to this entropy.

\subsection{$\alpha$-lazy outward random walk}

Let $G=(V,E,w)$ be a connected weighted finite graph, where $V=\{x_1,x_2,\cdots,x_n\}$ is the vertex set, $E=\{e_1,e_2,\cdots,e_m\}$ is the edge set and $w: E\rightarrow (0,+\infty)$ is
the weight function on $E$.
For any $x,y\in V$, we denote $x\sim y$ if $xy\in E$. For any set $A\subset V$ and a vertex $x\in V$, we say
$x\sim A$ if there is a vertex $a\in A$ such that $x\sim a$ and $x\not\in A$.

 A $1$-step neighbor set with respect to $x$ refers to the set composed of all
 $1$-step neighbors of $x$, namely
 $$N_1(x)=\left\{u\in V: u\sim N_0(x)\right\},$$
 where $N_0(x)=\{x\}$;
 Inductively, a $j$-step neighbor set with respect to $x$ refers to the set composed of all
 $j$-step neighbors of $x$, namely
$$N_j(x)=\left\{u\in V: u\sim \cup_{k=0}^{j-1}N_k(x)\right\}, \forall j\geq 2.$$
Clearly $N_i(x)\cap N_j(x)=\varnothing$ for any $j\not=i$, and $N_j(x)=\varnothing$ if $j\geq n$ or $j\geq m+1$.  For clarity,
we set
$$n_x=\max\{j:N_j(x)\not=\varnothing\}.$$
Obviously, $V$ can be decomposed into the union of $n_x$ subsets $N_j$.

Fix $x\in V$ and $z\in N_\ell(x)$ for some integer $\ell$, we define a set of all outward paths from $x$ to $z$ by
$$\mathbf{\Gamma}_{\overrightarrow{xz}}^\ell=\left\{\gamma: x=z_0\rightarrow z_1\rightarrow z_2\rightarrow \cdots\rightarrow z_\ell=z \,\left|\, z_j\in N_j(x),\, z_{j}z_{j+1}\in E,\, \forall 1\leq j\leq \ell-1 \right\}.\right.$$
If $z\in N_1(x)$, to walk from $x$ to $z$, it is only allowed to move along the paths in $\mathbf{\Gamma}_{\overrightarrow{xz}}^1$. Thus the probability of walking from $x$ to $z$ is assumed to be
$$\mathscr{P}(x,z)=\frac{w_{xz}}{\sum_{u\in N_1(x)}w_{xu}}.$$
If $z\in N_\ell(x)$,  considering only paths in $\mathbf{\Gamma}_{\overrightarrow{xz}}^\ell$,
the probability of walking from $x$ to $z$  is
$$\mathscr{P}(x,z)=\sum_{z_1\in N_1(x),\cdots,z_{\ell-1}\in N_{\ell-1}(x)}\frac{w_{xz_1}}{\sum_{u\in N_1(x)}w_{xu}}\frac{w_{z_1z_2}}{\sum_{u\in N_2(x)}w_{z_1u}}\cdots\frac{w_{z_{\ell-1}z}}
{\sum_{u\in N_\ell(x)}w_{z_{\ell-1}u}},$$
where $z_0=x$, $z_\ell=z$, and in the summation terms, $w_{z_{j-1}z_j}=0$ if $z_{j-1}$ is not adjacent to $z_j$.

Given $\alpha\in (0,1)$, we shall construct a special $\alpha$-lazy outward random walk. Assume the  probability that
node $x$ remains stationary is $\alpha$. If $z\in N_1(x)$ and $z\sim N_2(x)$, then the probability of walking from $x$ to $z$
is assumed to be $(1-\alpha)\alpha\mathscr{P}(x,z)$. While if $z\in N_1(x)$ and $z\not\sim N_2(x)$,
then the probability of walking from $x$ to $z$ is
assumed to be $(1-\alpha)\mathscr{P}(x,z)$. Here and in the sequel, $z\not\sim N_2(x)$ means $z$ is not a neighbor of the set $N_2(x)$.
In general, if $z\in N_j(x)$ and $z\sim N_{j+1}(x)$, then the probability of walking from $x$ to $z$, along any possible outward path
$\gamma\in\mathbf{\Gamma}_{\overrightarrow{xz}}^j$, is assigned a value
$(1-\alpha)^j\alpha\mathscr{P}(x,z)$;
if $z\in N_j(x)$ and $z\not\sim N_{j+1}(x)$, then the probability of walking from $x$ to $z$, along any possible outward path
$\gamma\in\mathbf{\Gamma}_{\overrightarrow{xz}}^j$, is assigned a value $(1-\alpha)^j\mathscr{P}(x,z)$.

To summarize the above discussion, we have the following definition:

\vspace{0.3cm}
\noindent{\bf Definition A.} Fix $\alpha\in (0,1)$ and $x\in V$. An {\it $\alpha$-lazy outward random walk}
$\textsf{R}_x^\alpha$ is defined as follows:
\begin{equation}\label{distribution}
\textsf{R}_x^\alpha(z)=\left\{
\begin{array}{lll}
\alpha &{\rm if}& z\in N_0(x)=\{x\}\\[1.5ex]
(1-\alpha)^j\alpha\mathscr{P}(x,z) &{\rm if}& z\in N_j(x),\,z\sim N_{j+1}(x)\\[1.5ex]
(1-\alpha)^j\mathscr{P}(x,z) &{\rm if}& z\in N_j(x),\,z\not\sim N_{j+1}(x)\\[1.5ex]
j=1,2,\cdots,n_x.
\end{array}
\right.
\end{equation}

It is easy to check that for any fixed number $\alpha\in (0,1)$, there hold
\begin{equation*}
\textsf{R}_x^\alpha(z)>0,\quad\forall x,z\in V
\end{equation*}
and
\begin{equation*}
\sum_{z\in V}\textsf{R}_x^\alpha(z)=1,\quad\forall x\in V.
\end{equation*}

\subsection{The KL divergence}
The Kullback-Leibler (KL) divergence, which is also called the relative entropy,  measures how one probability distribution diverges from a second, reference probability
distribution. It is widely used in statistics, machine learning, and information theory. In the graph setting, for two
probability measures $P$ and $Q$ on $V$, the KL divergence reads
$$\mathcal{D}_{KL}(P,Q)=\sum_{z\in V}P(z)\log\frac{P(z)}{Q(z)}.$$
According to  \cite{Kullback1951}, the KL divergence is non-symmetric and non-negative.
Fix any two nodes $x,y\in V$. Coming back to the $\alpha$-lazy outward random walks $\textsf{R}_x^\alpha$ and $\textsf{R}_y^\alpha$,
we have their KL divergence as
\begin{equation}\label{entropy}\mathcal{D}_{KL}(\textsf{R}_x^\alpha,\textsf{R}_y^\alpha)=\sum_{z\in V}\textsf{R}_x^\alpha(z)\log\frac{\textsf{R}_x^\alpha(z)}{\textsf{R}_y^\alpha(z)}.\end{equation}

\subsection{The entropy flow}
Let $\mathfrak{D}:(0,1)\times E\rightarrow \mathbb{R}$ be a function defined by
\begin{equation}\label{entropy-alpha}\mathfrak{D}_{e}^\alpha=\mathfrak{D}(\alpha,e)=\mathcal{D}_{KL}(\textsf{R}_x^\alpha,\textsf{R}_y^\alpha)+
\mathcal{D}_{KL}(\textsf{R}_y^\alpha,\textsf{R}_x^\alpha),\end{equation}
where $e=xy\in E$, $\mathcal{D}_{KL}(\cdot,\cdot)$ is the KL divergence as in (\ref{entropy}).
We propose an entropy flow
\begin{equation}\label{entropy-flow}
\left\{\begin{array}{lll}w_e^\prime(t)=\mathfrak{D}_e^\alpha(t),&& t>0\\[1.5ex]
w_e(t)>0, && t>0\\[1.5ex]w_e(0)=w_{0,e},&&\forall e\in E,\end{array}\right.
\end{equation}
where $\mathfrak{D}_e^\alpha(t)$ denotes the entropy $\mathfrak{D}_e^\alpha$ with respect to the weight $\mathbf{w}(t)=
(w_{e_1}(t),\cdots,w_{e_m}(t))$.
Hereafter, to simplify notations, we use the name "entropy" to denote any related entropies including KL divergence,
relative entropy and others.

\vspace{0.3cm}
Our main theoretical result is the following:

\begin{theorem}\label{existence}
Let $\alpha\in(0,1)$ and $\mathbf{w}_0=(w_{0,e_1},\cdots,w_{0,e_m})\in\mathbb{R}^m_+=\{\mathbf{w}=(w_{e_1},\cdots,w_{e_m})\in\mathbb{R}^m: w_{e_j}>0,j=1,\cdots,m\}$ be fixed. Then the entropy flow (\ref{entropy-flow}) has a unique solution $\mathbf{w}(t)=(w_{e_1}(t),\cdots,w_{e_m}(t))$ for $t\in[0,+\infty)$. Furthermore, we have either $(i)$ there holds
$$\lim_{t\rightarrow+\infty}w_{e_j}(t)=+\infty \quad{\rm for\,\,all}\,\, 1\leq j\leq m,$$
or $(ii)$ there exist constants $w_{e_j}^\ast>0$ such that
$$\lim_{t\rightarrow +\infty}w_{e_j}(t)=w_{e_j}^\ast,\,\,\lim_{t\rightarrow +\infty}\mathfrak{D}_{e_j}^\alpha(t)=\mathfrak{D}_{e_j}^\ast=0\quad{\rm for\,\,all}\,\, 1\leq j\leq m,$$
where $\mathfrak{D}^\alpha_{e_j}(t)$ and $\mathfrak{D}_{e_j}^\ast$ are the entropies on $e_j$ with respect to the weights
$\mathbf{w}(t)$ and $\mathbf{w}^\ast=(w_{e_1}^\ast,\cdots,w_{e_m}^\ast)$ respectively.
\end{theorem}

As a consequence of Theorem \ref{existence}, we have

\begin{corollary}\label{Corollary}
Let $G=(V,E,\mathbf{w}_0)$ be a triangle, where $V=\{x,y,z\}$ is the vertex set, $E=\{xy,yz,xz\}$ is the edge set, and
$\mathbf{w}_0=(w_{0,xy},w_{0,yz},w_{0,xz})$ is the initial weight on $E$. Let $\alpha\in(0,1)$ and
$\mathbf{w}(t)=(w_{xy}(t),w_{yz}(t),w_{xz}(t))$ be the unique global solution of the entropy flow $w_e^\prime(t)=\mathfrak{D}_e^\alpha(t)$
with $\mathbf{w}(0)=\mathbf{w}_0$. If $\alpha\not=1/3$, then ${w}_e(t)\rightarrow +\infty$ for all $e\in E$.
\end{corollary}

\begin{remark}
Similar to the entropy flow (\ref{entropy-flow}), one can also consider
\begin{equation}\label{entropy-flow-2}
\left\{\begin{array}{lll}w_{xy}^\prime(t)=\mathcal{D}_{KL}(\mathsf{R}_x^\alpha(t),\mathsf{R}_y^\alpha(t)),&& t>0\\[1.5ex]
w_{xy}(t)>0, && t>0\\[1.5ex]w_{xy}(0)=w_{0,xy},&&\forall xy\in E\end{array}\right.
\end{equation}
or
\begin{equation}\label{entropy-flow-3}
\left\{\begin{array}{lll}w_{xy}^\prime(t)=\mathcal{D}_{KL}(\mathsf{R}_y^\alpha(t),\mathsf{R}_x^\alpha(t)),&& t>0\\[1.5ex]
w_{xy}(t)>0, && t>0\\[1.5ex]w_{xy}(0)=w_{0,xy},&&\forall xy\in E.\end{array}\right.
\end{equation}
Completely analogous to Theorem \ref{existence}, one also has the existence and uniqueness
of solutions to (\ref{entropy-flow-2}) and (\ref{entropy-flow-3}). Also, these two entropy flows can be applied to community detection,
with performance compatible with that of (\ref{entropy-flow}). We omitted the details but leave them to interested readers.
\end{remark}

We shall give several examples of entropy flow, and explain why the entropy flow
can be used to detect communities in Section \ref{Example}; The proof of Theorem \ref{existence} and Corollary \ref{Corollary} is based on the theory of ODE, and will be
given in Section \ref{proof}; As an application of Theorem \ref{existence}, community detection will be discussed in Section \ref{application}.
What we didn't expect was that entropy flow
could also be used to solve the problem of community detection, where it performs as well as curvature flow based on Ollivier's Ricci curvature or Lin-Lu-Yau's Ricci curvature.
Throughout this paper, we often denote various constants by the same $C$ from line to line, even in the same line. The readers can distinguish it
from the context.

\section{Examples of entropy flow}\label{Example}
In this section, we give some examples of the entropy flow (\ref{entropy-flow}). In some special cases, solutions of
the entropy flows can be written explicitly.

\begin{example} For the regular polygon graph, we can solve the entropy flow explicitly, say\\
1) $G = (V,E,\mathbf{w}_0)$ is a line segment, where $V = \{x,y\}$, $E = \{xy\}$, and $\mathbf{w}_0=w_0$.
Assume $w(t)$ is the unique solution of the entropy flow (\ref{entropy-flow}), $t\in[0,+\infty)$.  Note that for
$\alpha \in (0,1)$, the $\alpha$-lazy outward random walks with respect to $w(t)$ read as
$$
\mathsf{R}_x^\alpha(u,t) = \begin{cases} \alpha, & u=x,\\ 1-\alpha, & u=y, \end{cases} \qquad\qquad
\mathsf{R}_y^\alpha(u,t) = \begin{cases} 1-\alpha, & u=x,\\ \alpha, & u=y. \end{cases}
$$
Thus the entropy
$$\mathfrak{D}_{xy}^\alpha(t)=\mathcal{D}_{KL}(\mathsf{R}_x^\alpha(t), \mathsf{R}_y^\alpha(t))
+ \mathcal{D}_{KL}(\mathsf{R}_y^\alpha(t), \mathsf{R}_x^\alpha(t))=(2-4\alpha) \log \frac{1-\alpha}{\alpha},$$
and whence
$$w(t)=w_0+\left((2-4\alpha) \log \frac{1-\alpha}{\alpha}\right)t.$$
2) $G = (V,E,\mathbf{w}_0)$ is a triangle,
where $V = \{x,y,z\}$, $E = \{xy, yz, zx\}$ and $\mathbf{w}_0 = (w_0, w_0, w_0)$.
We hope to find a solution of (\ref{entropy-flow}) such that it has the form $\mathbf{w}(t)=(w(t),w(t),w(t))$,
i.e. weights on edges are the same. Under this assumption,
for $\alpha \in (0,1)$ and $x,y\in V$, the $\alpha$-lazy outward random walk with respect to $\mathbf{w}(t)$ is written as
\begin{equation*}
\begin{aligned}
\mathsf{R}_x^\alpha(u,t) &=
\begin{cases}
\alpha, & u=x,\\
\frac{1-\alpha}{2}, & u=y,\\
\frac{1-\alpha}{2}, & u=z,
\end{cases} \qquad \qquad
\mathsf{R}_y^\alpha(u,t) &=
\begin{cases}
\frac{1-\alpha}{2}, & u=x,\\
\alpha, & u=y,\\
\frac{1-\alpha}{2}, & u=z.
\end{cases}
\end{aligned}
\end{equation*}
It follows that
$$\mathfrak{D}_{xy}^\alpha(t)=\mathcal{D}_{KL}(\mathsf{R}_x^\alpha(t), \mathsf{R}_y^\alpha(t))
+ \mathcal{D}_{KL}(\mathsf{R}_y^\alpha(t), \mathsf{R}_x^\alpha(t))
= (3\alpha-1)\log\frac{2\alpha}{1-\alpha}$$
and that
$$w(t)=w_0+\left((3\alpha-1)\log\frac{2\alpha}{1-\alpha}\right)t,\quad\forall t\in[0,+\infty).$$
By Theorem \ref{existence}, the solution of (\ref{entropy-flow}) is unique, which implies $\mathbf{w}(t)=(w(t),w(t),w(t))$ is
the exact unique solution. \\
3) $G = (V,E,\mathbf{w}_0)$ is a general regular polygon. Using the method in 2), one can easily write out the unique solution to the entropy flow.
\end{example}

Next, we present an example to illustrate how the entropy flow can be used for community detection.

\begin{example}
Let \( G=(V,E,\mathbf{w}_0) \) be a graph described as in Figure \ref{fig-entropy-flow-1}. The weight on each edge is assumed to be $1$.
Take $\alpha=0.5$. The entropy on each edge with respect to the initial weight $\mathbf{w}_0$ is calculated as
$\mathfrak{D}_{x_3x_4}^\alpha(0) = 2.19$, $\mathfrak{D}_{x_1x_2}^\alpha(0)
= \mathfrak{D}_{x_5x_6}^\alpha(0) = 0.35$,
$\mathfrak{D}_{x_1x_3}^\alpha(0)
= \mathfrak{D}_{x_2x_3}^\alpha(0)
= \mathfrak{D}_{x_4x_5}^\alpha(0)
= \mathfrak{D}_{x_4x_6}^\alpha(0)= 0.93$.
Choose $s=0.1$, $t_j=js$, $j=0,1,2,\cdots$. One discrete versions of (\ref{entropy-flow}) says
$$w_e(t_j)=
w_e(t_{j-1})
+
s\,\mathfrak{D}_e^\alpha(t_{j-1})=w_e(0)
+
s \sum_{k=0}^{j-1}
\mathfrak{D}_e^\alpha(t_k).
$$
At $t_{10}$, the edge weights become
$w_{x_3x_4}(t_{10}) = 2.80$, $w_{x_1x_2}(t_{10})
= w_{x_5x_6}(t_{10}) = 1.43$,
$w_{x_1x_3}(t_{10})= w_{x_2x_3}(t_{10})= w_{x_4x_5}(t_{10})
= w_{x_4x_6}(t_{10}) = 1.94$.
Note that $w_{x_3x_4}$ grows significantly faster than all
other six edge weights. Deleting $x_3x_4$, one gets two communities as the final graph.

\begin{figure}[h]
\centering

\definecolor{highE}{RGB}{200,50,50}   
\definecolor{midE}{RGB}{50,80,200}    
\definecolor{lowE}{RGB}{70,70,70}     

\begin{tikzpicture}[scale=1.1]

\begin{scope}
\coordinate (x1) at (-0.5,0.5);
\coordinate (x2) at (-0.5,-0.5);
\coordinate (x3) at (0.2,0);
\coordinate (x4) at (1.3,0);
\coordinate (x5) at (2,0.5);
\coordinate (x6) at (2,-0.5);

\draw[lowE, thick] (x1)--(x2)
    node[midway,left] {$w=1$};
\draw[lowE, thick] (x5)--(x6);

\draw[midE, thick] (x1)--(x3)
    node[midway,right=3pt] {$w=1$};
\draw[midE, thick] (x2)--(x3);
\draw[midE, thick] (x4)--(x5);
\draw[midE, thick] (x4)--(x6);

\draw[highE, very thick] (x3)--(x4)
    node[midway,below] {$w=1$};

\foreach \p in {x1,x2,x3,x4,x5,x6}
    \fill (\p) circle (1.5pt);

\node[above left] at (x1) {$x_1$};
\node[below left] at (x2) {$x_2$};
\node[below] at (x3) {$x_3$};
\node[below] at (x4) {$x_4$};
\node[above right] at (x5) {$x_5$};
\node[below right] at (x6) {$x_6$};
\end{scope}

\draw[->, thick] (3.3,0) -- (4.2,0)
    node[midway,above] {\scriptsize entropy flow};

\begin{scope}[shift={(6.2,0)}]
\coordinate (y1) at (-1,0.8);
\coordinate (y2) at (-1,-0.8);
\coordinate (y3) at (-0.1,0);
\coordinate (y4) at (2.1,0);
\coordinate (y5) at (3,0.8);
\coordinate (y6) at (3,-0.8);

\draw[lowE, thick] (y1)--(y2);
\draw[lowE, thick] (y5)--(y6)
 node[midway,right] {$w=1.43$};

\draw[midE, thick] (y1)--(y3)
    node[midway,right=3pt] {$w=1.94$};
\draw[midE, thick] (y2)--(y3);
\draw[midE, thick] (y4)--(y5);
\draw[midE, thick] (y4)--(y6);

\draw[highE, very thick] (y3)--(y4)
    node[midway,below] {$w=2.80$};

\foreach \p in {y1,y2,y3,y4,y5,y6}
    \fill (\p) circle (1.5pt);

\node[above left] at (y1) {$x_1$};
\node[below left] at (y2) {$x_2$};
\node[below] at (y3) {$x_3$};
\node[below] at (y4) {$x_4$};
\node[above right] at (y5) {$x_5$};
\node[below right] at (y6) {$x_6$};
\end{scope}

\end{tikzpicture}

\vspace{0.6cm}

\begin{tikzpicture}[scale=1.1]

\draw[->, thick] (-1.2,0) -- (0.8,0)
    node[midway,above] {\scriptsize surgery};

\begin{scope}[shift={(2.5,0)}]
\coordinate (z1) at (-1,0.8);
\coordinate (z2) at (-1,-0.8);
\coordinate (z3) at (-0.1,0);
\coordinate (z4) at (2.1,0);
\coordinate (z5) at (3,0.8);
\coordinate (z6) at (3,-0.8);

\draw (z1)--(z2) (z2)--(z3) (z3)--(z1);
\draw (z4)--(z5) (z5)--(z6) (z6)--(z4);

\foreach \p in {z1,z2,z3,z4,z5,z6}
    \fill (\p) circle (1.5pt);

\node[above left] at (z1) {$x_1$};
\node[below left] at (z2) {$x_2$};
\node[below] at (z3) {$x_3$};
\node[below] at (z4) {$x_4$};
\node[above right] at (z5) {$x_5$};
\node[below right] at (z6) {$x_6$};
\end{scope}

\end{tikzpicture}

\caption{community detection}
\label{fig-entropy-flow-1}
\end{figure}
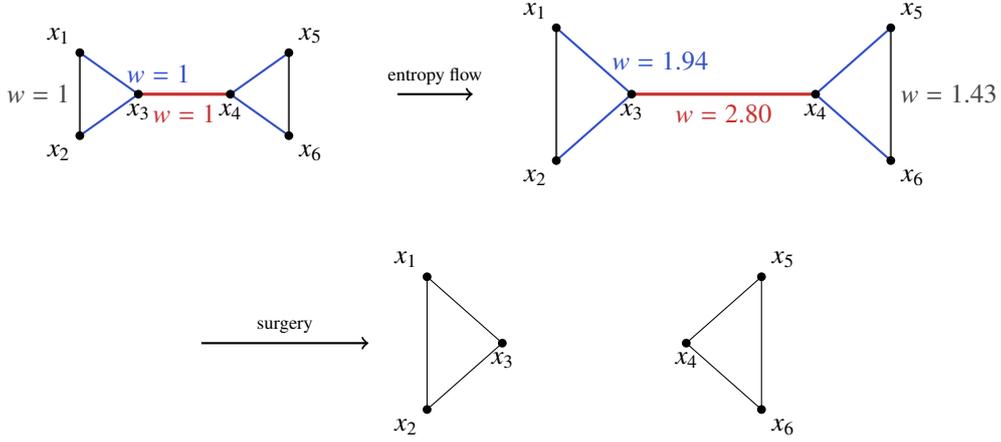

\end{example}

\section{Proof of the main theorem}\label{proof}

In this section, we will prove long time existence of the entropy flow (\ref{entropy-flow}) and the convergence of the weights along the entropy flow. This is based on the ODE theory. We will also provide the proof of Corollary \ref{Corollary}.

\vspace{0.3cm}
{\it Proof of Theorem \ref{existence}}. We divide the proof into several steps.
\vspace{0.3cm}

{\it Step 1. There holds $\mathcal{D}_{KL}(\textsf{R}_x^\alpha,\textsf{R}_y^\alpha)\geq 0$ for all $x,y\in V$ and all $\alpha\in(0,1)$. Moreover,
$\mathcal{D}_{KL}(\textsf{R}_x^\alpha,\textsf{R}_y^\alpha)= 0$ if and only if $\textsf{R}_x^\alpha(z)=\textsf{R}_y^\alpha(z)$ for all $z\in V$.}

The proof of this step is standard, we provide it here for readers' convenience.
Denote $V=\{z_1,z_2,\cdots,z_n\}$, $\theta_j=\textsf{R}_x^\alpha(z_j)$, $t_j={\textsf{R}_y^\alpha(z_j)}/{\textsf{R}_x^\alpha(z_j)}$,
$j=1,2,\cdots,n$. Since all $\theta_j>0$, $t_j>0$, $\theta_jt_j=\textsf{R}_y^\alpha(z_j)$, $\sum_{j=1}^n\theta_j=1$, $\sum_{j=1}^n\textsf{R}_y^\alpha(z_j)=1$ and $-\log t$ is a convex function of one variable $t\in(0,+\infty)$, we have
\begin{eqnarray*}
\mathcal{D}_{KL}(\textsf{R}_x^\alpha,\textsf{R}_y^\alpha)&=&-\sum_{j=1}^n\theta_j\log t_j\\
&\geq&-\log\left(\sum_{j=1}^n\theta_jt_j\right)\\
&=&-\log\left(\sum_{j=1}^n\textsf{R}_y^\alpha(z_j)\right)\\
&=&0.
\end{eqnarray*}
Clearly, the above equality holds if and only if $t_1=t_2=\cdots=t_n$, or equivalently $\textsf{R}_x^\alpha(z)=\textsf{R}_y^\alpha(z)$ for all $z\in V$.\\

{\it Step 2. Let $\mathbf{w}(t)=(w_{e_1}(t),\cdots,w_{e_m}(t))$ be a solution of the entropy flow (\ref{entropy-flow})
for $t\in [0,T)$. Then, on any edge $e\in E$, the weight $w_e(t)$ is increasing in $t\in [0,T)$.}

 Denote $\mathscr{P}$ as the set of all probability measures and $\textsf{R}_x^\alpha: [0,T)\times V\rightarrow \mathscr{P}$ be a map satisfying $\textsf{R}_x^\alpha(t,z)=\textsf{R}_x^\alpha(z)$, given as in (\ref{distribution}), where $\mathbf{w}=(w_{e_1},\cdots,w_{e_m})$ is replaced by ${\mathbf{w}}(t)$.
 By Step 1, we have
$$w_e^\prime(t)=\mathfrak{D}_e^\alpha(t)=\mathcal{D}_{KL}(\textsf{R}_x^\alpha(t,\cdot),\textsf{R}_y^\alpha(t,\cdot))+
\mathcal{D}_{KL}(\textsf{R}_y^\alpha(t,\cdot),\textsf{R}_x^\alpha(t,\cdot))\geq 0.$$
This gives the monotonicity of $w_e(t)$ with respect to $t\in[0,T)$.\\

{\it Step 3. Short time existence.}

Denote a vector valued function $\mathbf{F}:\mathbb{R}^m_+\rightarrow\mathbb{R}^m$ by
$$\mathbf{F}(\mathbf{w})=(F_1(\mathbf{w}),\cdots,F_m(\mathbf{w}))^\top,$$
where $\mathbb{R}^m_+=\{\mathbf{w}=(w_{1},\cdots,w_{m})\in\mathbb{R}^m: w_{j}>0,j=1,\cdots,m\}$ and
 $F_j(\mathbf{w})=\mathfrak{D}_{e_j}^\alpha$. It suffices to prove that $\mathbf{F}$ is locally Lipschitz in $\mathbb{R}^m_+$. To see this, we fix any
$\alpha\in(0,1)$ and $x\in V$. Then for any two vectors
$\mathbf{w},\widetilde{\mathbf{w}}\in \overline{\Omega} \subset\subset\mathbb{R}^m_+$, we take two positive $\alpha$-lazy random walks $\textsf{R}_x^\alpha$ and $\widetilde{\textsf{R}}_x^\alpha$ determined by $\mathbf{w}$ and $\widetilde{\mathbf{w}}\in \mathbb{R}^m_+$ respectively.
With no loss of generality, we assume $n_x\geq 2$ and $\epsilon\leq w_{e}\leq 1/\epsilon$ for all $e\in E$, where $\epsilon>0$ is a constant
depending only on the distance between $\overline\Omega$ and $\partial\mathbb{R}^m_+$. Obviously $\textsf{R}_x^\alpha(x)-\widetilde{\textsf{R}}_x^\alpha(x)=0$. While if $z\in N_j(x)$ for
$1\leq j\leq n_x$, then
\begin{eqnarray*}
|\textsf{R}_x^\alpha(z)-\widetilde{\textsf{R}}_x^\alpha(z)|&\leq&\sum_{u_1\in N_1(x),\cdots,u_{j-1}\in N_{j-1}(x)}\prod_{k=1}^j
\left|\frac{w_{u_{k-1}u_k}}{\sum_{u\in N_k(x)}w_{u_{k-1}u}}-\frac{\widetilde{w}_{u_{k-1}u_k}}{\sum_{u\in N_k(x)}\widetilde{w}_{u_{k-1}u}}\right|
\\[1.5ex]&\leq&C|\mathbf{w}-\widetilde{\mathbf{w}}|
\end{eqnarray*}
for some constant $C$ depending only on $m$, $n$ and $\epsilon$. Clearly there exists a constant $C$ depending only on $m$, $n$ and $\epsilon$
such that
$$1/C\leq \textsf{R}_x^\alpha(z), \widetilde{\textsf{R}}_{x}(z), \textsf{R}_y^\alpha(z), \widetilde{\textsf{R}}_{y}(z)\leq C,\quad \forall z\in V,\,\forall e=xy\in E.$$
It follows that
\begin{eqnarray*}
\left|\mathcal{D}_{KL}(\textsf{R}_x^\alpha,\textsf{R}_y^\alpha)-\mathcal{D}_{KL}(\widetilde{\textsf{R}}_x^\alpha,
\widetilde{\textsf{R}}_y^\alpha)\right|&\leq&
\sum_{z\in V}\left|\textsf{R}_x^\alpha(z)\log\frac{\textsf{R}_x^\alpha(z)}{\textsf{R}_y^\alpha(z)}-
\widetilde{\textsf{R}}_x^\alpha(z)\log\frac{\widetilde{\textsf{R}}_x^\alpha(z)}{\widetilde{\textsf{R}}_y^\alpha(z)}\right|\\[1.5ex]
&\leq&C\sum_{z\in V}\left(|\textsf{R}_x^\alpha(z)-\widetilde{\textsf{R}}_x^\alpha(z)|+|\textsf{R}_y^\alpha(z)-\widetilde{\textsf{R}}_y^\alpha(z)|\right)\\[1.5ex]
&\leq&C|\mathbf{w}-\widetilde{\mathbf{w}}|.
\end{eqnarray*}
In the same way, $$|\mathcal{D}_{KL}(\textsf{R}_y^\alpha,\textsf{R}_x^\alpha)-\mathcal{D}_{KL}(\widetilde{\textsf{R}}_y^\alpha,\widetilde{\textsf{R}}_x^\alpha)|\leq
C|\mathbf{w}-\widetilde{\mathbf{w}}|.$$
Hence
$$|F_j(\mathbf{w})-F_j(\widetilde{\mathbf{w}})|\leq C|\mathbf{w}-\widetilde{\mathbf{w}}|, \quad\forall 1\leq j\leq m,$$
and thus $\mathbf{F}(\mathbf{w})$ is locally Lipschitz in $\mathbb{R}^m_+$.
Then we conclude from the ODE theory that the entropy flow (\ref{entropy-flow}) has a unique solution on $[0,T_0]$ for some $T_0>0$. \\

{\it Step 4. Long time existence.}

Based on the third step, we set
$$T_\infty=\sup\{T>0: (\ref{entropy-flow})\,{\rm has\,a\,unique\,solution\,on}\, [0,T] \}.$$
It suffices to prove $T_\infty=+\infty$. For otherwise, we may assume $T_\infty<+\infty$. Let $\mathbf{w}(t)$ be a unique solution of
$$\left\{\begin{array}{lll}
\mathbf{w}^\prime(t)=\mathbf{F}(\mathbf{w}(t)),\,\,t\in[0,T_\infty)\\[1.5ex]
{w_e}(t)>0,\,\forall e\in E,\, t\in[0,T_\infty)\\[1.5ex]
{w_e}(0)={w}_{0,e},\,\forall e\in E.
\end{array}\right.$$
We first observe that for each $e=xy\in E$, the fact $w_e^\prime(t)=\mathfrak{D}_e^\alpha(t)\geq 0$ implies
$$w_e(t)\geq w_{0,e}>0,\quad \forall t\in[0,T_\infty).$$
Now we distinguish three cases to estimate the entropy $\mathfrak{D}_{e}^\alpha$ as follows:\\

$(i)$ If $z\in N_0(y)$, there holds
\begin{eqnarray*}
\textsf{R}_x^\alpha(z)\log\frac{\textsf{R}_x^\alpha(z)}{\textsf{R}_y^\alpha(z)}\leq \log\frac{1}{\textsf{R}_y^\alpha(z)}
=\log\frac{1}{\alpha},
\end{eqnarray*}
since $\textsf{R}_x^\alpha(z), \textsf{R}_y^\alpha(z)$ are probability measures satisfying
$$0<\textsf{R}_x^\alpha(z)<1,\,\, 0<\textsf{R}_y^\alpha(z)<1;$$

$(ii)$ If $z\in N_j(y)$ with $1\leq j\leq n_y-1$, recalling the definition of $N_j(y)$ (the $j$th step
neighbor set centered at the node $y$), one has
\begin{eqnarray*}
\textsf{R}_x^\alpha(z)\log\frac{\textsf{R}_x^\alpha(z)}{\textsf{R}_y^\alpha(z)}&\leq&\log\frac{1}{\textsf{R}_y^\alpha(z)}\\
&=&\log\frac{1}{(1-\alpha)^j\alpha}+\log\frac{1}{\sum_{u_1\in N_1(y),\cdots,u_{j-1}\in N_{j-1}(y)}
\prod_{k=1}^j\frac{w_{u_{k-1}u_k}}{\sum_{u\in N_k(y)}w_{u_{k-1}u}}}\\
&\leq&\log\frac{1}{(1-\alpha)^j\alpha}+\log\frac{1}{\prod_{k=1}^j\frac{w_{u_{k-1}u_k}}{\sum_{u\in N_k(y)}w_{u_{k-1}u}}}\\
&\leq&\log\frac{1}{(1-\alpha)^j\alpha}+\sum_{k=1}^j\log\frac{\sum_{u\in N_k(y)}w_{u_{k-1}u}}{w_{u_{k-1}u_k}}\\
&\leq&\log\frac{1}{(1-\alpha)^j\alpha}+j\log\frac{\sum_{\tau\in E}w_\tau}{\min_{\tau\in E} w_{0,\tau}},
\end{eqnarray*}
where the second inequality holds for any fixed $u_1\in N_1(y),\cdots,u_{j-1}\in N_{j-1}(y)$.

$(iii)$ If
$z\in N_{n_y}(y)$, we have
$$\textsf{R}_x^\alpha(z)\log\frac{\textsf{R}_x^\alpha(z)}{\textsf{R}_y^\alpha(z)}\leq
\log\frac{1}{(1-\alpha)^{n_y}}+n_y\log\frac{\sum_{\tau\in E}w_\tau}{\min_{\tau\in E} w_{0,\tau}}.$$
Combining $(i)$, $(ii)$ and $(iii)$, we find some constant $C$ depending only on $m$, $n$ and $\alpha$ such that
\begin{eqnarray*}\mathcal{D}_{KL}(\textsf{R}_x^\alpha,\textsf{R}_y^\alpha)&=&\sum_{z\in V}\textsf{R}_x^\alpha(z)\log\frac{\textsf{R}_x^\alpha(z)}{\textsf{R}_y^\alpha(z)}\\
&\leq&C\left(1+\log \sum_{\tau\in E}w_\tau\right).\end{eqnarray*}
In the same way,
$$\mathcal{D}_{KL}(\textsf{R}_y^\alpha,\textsf{R}_x^\alpha)\leq C\left(1+\log \sum_{\tau\in E}w_\tau\right).$$
Hence
$$\mathfrak{D}_e^\alpha=\mathcal{D}_{KL}(\textsf{R}_x^\alpha,\textsf{R}_y^\alpha)+\mathcal{D}_{KL}(\textsf{R}_y^\alpha,\textsf{R}_x^\alpha)\leq
C\left(1+\log \sum_{\tau\in E}w_\tau\right).$$
Coming back to (\ref{entropy-flow}), we obtain
$$\frac{d}{dt}\left(\sum_{\tau\in E}w_\tau\right)\leq C\left(1+\log\sum_{\tau\in E}w_\tau\right)\leq
C\left(1+\sum_{\tau\in E}w_\tau\right).$$
It follows that
$$0<\min_{\tau\in E}w_{0,\tau}\leq w_e(t)\leq\sum_{\tau\in E}w_\tau(t)\leq Ce^{Ct}.$$
Thus $\mathbf{w}(t)$ is bounded in $\mathbb{R}^m_+$ with respect to $t\in [0,T_\infty)$, and $w_\tau(t)\geq w_{0,\tau}>0$ for all $t\in[0,T_\infty)$. By the ODE theory, solution $\mathbf{w}(t)$ can be extended to $[0,T_1)$ for some $T_1>T_\infty$.
This contradicts the definition of $T_\infty$, and confirms the long time existence of $\mathbf{w}(t)$.

\vspace{0.3cm}
{\it Step 5. Convergence.}

We still have the second part of the theorem  (convergence of solutions) left. By Steps 2 and 4, for each edge $\tau\in E$, $w_\tau(t)$ is increasing in $t\in[0,+\infty)$. Hence there holds for each $\tau\in E$,
\begin{equation}\label{w-infinity}\lim_{t\rightarrow +\infty}w_\tau(t)=+\infty\end{equation}
or there exists a positive constant $w_\tau^\ast$ such that \begin{equation}\label{w-bdd}\lim_{t\rightarrow+\infty}w_\tau(t)=w_\tau^\ast.\end{equation}
We now exclude the possibility that there exist two edges $\tau_1$ and $\tau_2$ satisfying (\ref{w-infinity}) and (\ref{w-bdd}) respectively.
To see this, without loss of generality, we assume $\tau_1=yz$ and $\tau_2=xy$ are adjacent, $w_{\tau_1}(t)\rightarrow+\infty$ and
$w_{\tau_2}(t)\rightarrow w_{\tau_2}^\ast$ as $t\rightarrow+\infty$.
 Since
$\textsf{R}_x^\alpha(x,t)=\alpha$ and
$$\textsf{R}_y^\alpha(x,t)=\left\{\begin{array}{lll}(1-\alpha)\frac{w_{xy}(t)}{\sum_{u\in N_1(y)}w_{yu}(t)}&{\rm if}&
x\not\sim N_2(y)\\[1.5ex](1-\alpha)\alpha\frac{w_{xy}(t)}{\sum_{u\in N_1(y)}w_{yu}(t)}&{\rm if}&
x\sim N_2(y),\end{array}
\right.$$
we have by $w_{yz}(t)\rightarrow+\infty$ that
\begin{equation}\label{tend-0}
\lim_{t\rightarrow+\infty}\textsf{R}_y^\alpha(x,t)=0.
\end{equation}
By an elementary inequality $a\log a\geq -e^{-1}$, $\forall a>0$,
a calculation shows
\begin{eqnarray}\nonumber
\mathfrak{D}_{\tau_2}(t)&=&\sum_{u\in V}\textsf{R}_x^\alpha(u,t)\log\frac{\textsf{R}_x^\alpha(u,t)}{\textsf{R}_y^\alpha(u,t)}+
\sum_{u\in V}\textsf{R}_y^\alpha(u,t)\log\frac{\textsf{R}_y^\alpha(u,t)}{\textsf{R}_x^\alpha(u,t)}\\\nonumber
&=&\textsf{R}_x^\alpha(x,t)\log\frac{\textsf{R}_x^\alpha(x,t)}{\textsf{R}_y^\alpha(x,t)}+\sum_{u\in V\setminus\{x\}}\textsf{R}_x^\alpha(u,t)\log\frac{\textsf{R}_x^\alpha(u,t)}{\textsf{R}_y^\alpha(u,t)}+\sum_{u\in V}\textsf{R}_y^\alpha(u,t)\log\frac{\textsf{R}_y^\alpha(u,t)}{\textsf{R}_x^\alpha(u,t)}\\\label{estim}
&\geq& \alpha\log\frac{\alpha}{\textsf{R}_y^\alpha(x,t)}-2(n-1)e^{-1}.
\end{eqnarray}
Combining (\ref{tend-0}), (\ref{estim}) and (\ref{entropy-flow}), we obtain
$$w_{\tau_2}^\prime(t)=\mathfrak{D}_{\tau_2}^\alpha(t)\rightarrow+\infty.$$
Hence there exists $T^\ast>0$ such that for all $t\geq T^\ast$, $w_{\tau_2}^\prime(t)>1$, and thus
$$w_{\tau_2}(t)-w_{\tau_2}(T^\ast)>t-T^\ast.$$
This contradicts $w_{\tau_2}(t)\rightarrow w_{\tau_2}^\ast$. Therefore either $(i)$ $w_\tau$ satisfies (\ref{w-infinity}) for all $\tau\in E$,
or $(ii)$ $w_\tau$ satisfies (\ref{w-bdd}) for all $\tau\in E$.

We still need to prove that under condition $(ii)$, there holds for each $\tau\in E$,
\begin{equation}\label{zero}
\lim_{t\rightarrow+\infty}\mathfrak{D}_\tau(t)=\mathfrak{D}_\tau^\ast=0,
\end{equation}
where $\mathfrak{D}_\tau^\ast$ stands for the entropy on the edge $\tau$ with respect to the weight
$\mathbf{w}^\ast=(w_{\tau_1}^\ast,\cdots,w_{\tau_m}^\ast)$. In view of Definition A, it follows from
$w_\tau(t)\in C^1[0,+\infty)$ and $w_\tau(t)\rightarrow w_\tau^\ast$ that $\mathfrak{D}_\tau(t)\rightarrow\mathfrak{D}_\tau^\ast$ as
$t\rightarrow+\infty$. Since $\mathfrak{D}_\tau^\ast\geq 0$, it suffices to prove $\mathfrak{D}_\tau^\ast>0$ does not hold. Suppose
$\mathfrak{D}_\tau^\ast>0$ for some $\tau\in E$. Then there exists some $T^{\ast\ast}>0$ such that for all $t\geq T^{\ast\ast}$,
$$w_\tau^\prime=\mathfrak{D}_\tau^\alpha(t)\geq \frac{1}{2}\mathfrak{D}_\tau^\ast.$$
This leads to
$$w_\tau(t)=w_\tau(T^{\ast\ast})+\int_{T^{\ast\ast}}^t\mathfrak{D}_\tau^\alpha(s)ds\geq \frac{1}{2}(t-T^{\ast\ast})\mathfrak{D}_\tau^\ast,$$
contradicting $w_\tau(t)\rightarrow w_\tau^\ast$ as $t\rightarrow+\infty$. As a consequence, (\ref{zero}) holds, as we desired. $\hfill\Box$\\

{\it Proof of Corollary \ref{Corollary}.} Let $\mathbf{w}(t)=(w_{xy}(t),w_{yz}(t), w_{xz}(t))$ be the unique global solution of the entropy flow.  Suppose there exists an edge $\tau$
such that $w_\tau(t)$ is bounded with respect to $t$. Then by Theorem \ref{existence}, $\mathbf{w}(t)
\rightarrow \mathbf{w}^\ast=(w_{xy}^\ast,w_{yz}^\ast,w_{xz}^\ast)$ and
$\mathbf{\mathfrak{D}}(t)=(\mathfrak{D}_{xy}^\alpha(t),\mathfrak{D}_{yz}^\alpha(t),\cdots,\mathfrak{D}_{xz}^\alpha(t))\rightarrow \mathbf{\mathfrak{D}}^\ast=
(\mathfrak{D}_{xy}^\ast,\mathfrak{D}_{yz}^\ast,\mathfrak{D}_{xz}^\ast)=(0,0,0)$ as $t\rightarrow+\infty$, where each $\mathfrak{D}_e^\ast$ is the entropy on edge $e$ with respect to the weight $\mathbf{w}^\ast$. Clearly,
$\alpha$-lazy outward random walks with respect to the limit weight $\mathbf{w}^\ast$ are represented by
\begin{eqnarray*}
\textsf{R}_x^\ast&=&\left(\alpha,(1-\alpha)\frac{w_{xy}^\ast}{w_{xy}^\ast+w_{xz}^\ast},
(1-\alpha)\frac{w_{xz}^\ast}{w_{xy}^\ast+w_{xz}^\ast}\right),\\
\textsf{R}_y^\ast&=&\left((1-\alpha)\frac{w_{xy}^\ast}{w_{xy}^\ast+w_{yz}^\ast},\alpha,(1-\alpha)\frac{w_{yz}^\ast}{w_{xy}^\ast+w_{yz}^\ast}\right),\\
\textsf{R}_z^\ast&=&\left((1-\alpha)\frac{w_{xz}^\ast}{w_{xz}^\ast+w_{yz}^\ast},(1-\alpha)\frac{w_{yz}^\ast}{w_{xz}^\ast+w_{yz}^\ast},\alpha\right).
\end{eqnarray*}
Note that $\mathfrak{D}^\ast=(0,0,0)$ if and only if $\textsf{R}_x^\ast=\textsf{R}_y^\ast=\textsf{R}_z^\ast$, which is equivalent to
$$w_{xy}^\ast=w_{xz}^\ast=w_{yz}^\ast,\,\alpha=1/3.$$
Hence, if $\alpha\not=1/3$, $w_\tau(t)\rightarrow +\infty$ for all $\tau\in E$ as $t\rightarrow+\infty$. $\hfill\Box$

\section{Applications}\label{application}
In this section, we demonstrate the application of the entropy flow (\ref{entropy-flow}) to community detection on real-world networks. In Subsection \ref{Algorithm},  we present the algorithm for the entropy flow; In Subsection \ref{experiment},
we introduce the experimental setup, including benchmark datasets and evaluation metrics; In Subsection \ref{iterations}, we investigate
how the number of iterations influences the performance of the entropy flow algorithm;
In Subsection \ref{entropydistribution}, we examine the distributions of edge entropy and edge weight before and after the entropy flow; In Subsection \ref{threshold}, we examine how varying the surgery threshold affects community structure through entropy flow; In Subsection \ref{comparison}, we compare our method with classical community detection algorithms and the Ricci curvature flow algorithms; Finally, in Subsection \ref{cost}, we evaluate the computational time of entropy flow in comparison with Ricci flow.

\subsection{Algorithm design}\label{Algorithm}

Let $G=(V,E,\mathbf{w}_0)$ be a connected weighted finite graph,
where $V$ is the vertex set, $E$ is the edge set, and $\mathbf{w}_0=(w_{0,e_1},\cdots,w_{0,e_m})$ is the weight on $E$.
We now discretize the entropy flow (\ref{entropy-flow}).
Take $s>0$ as a step size, $t_j=js$, $j=0,1,\cdots$. Let $\mathfrak{D}_e^\alpha(0)$ be the entropy defined as in
(\ref{entropy-alpha}) with respect to $\alpha\in(0,1)$ and the initial weight $\mathbf{w}_0$,
and $w_e(t_1)=w_e(0)+s\mathfrak{D}_e^\alpha(0)$. By induction, if we assume $\mathfrak{D}_e^\alpha(t_{j-1})$ is the entropy
with respect to the weight $w_e(t_{j-1})$, then $w_e(t_j)=w_e(t_{j-1})+s\mathfrak{D}_e^\alpha(t_{j-1})$. To sum up, we have
for all $e=xy\in E$ and $j=1,2,\cdots$,
\begin{equation}\label{weight-j}w_e(t_j)=w_{0,e}+s\sum_{k=1}^{j-1}\mathfrak{D}_e^\alpha(t_k),\end{equation}
and
\begin{equation}\label{entropy-j}\mathfrak{D}_e^\alpha(t_{j-1})=\mathcal{D}_{KL}(\textsf{R}_x^\alpha(t_{j-1}),\textsf{R}_y^\alpha(t_{j-1}))+
\mathcal{D}_{KL}(\textsf{R}_y^\alpha(t_{j-1}),\textsf{R}_x^\alpha(t_{j-1})).\end{equation}
If $G$ is not connected, then it is a union of connected components $G_1,G_2,\cdots,G_\ell$ for some integer $\ell\geq 2$.
Noting that the operations (\ref{weight-j}) and (\ref{entropy-j}) can be done simultaneously in all connected components $G_1$, $G_2$, $\cdots$,
$G_\ell$, one need not check whether the initial graph is connected or not  before the process of entropy flow iterations.
The pseudo-code for entropy flow  is as follows.\\

\begin{algorithm}[H]
\caption{Community detection using entropy flow}
\label{alg:entropy-flow}
\KwIn{An undirected connected finite graph $G=(V,E)$; initial weight $\mathbf{w}_0$; parameter $\alpha\in(0,1)$; maximum number of iterations $N$; step size $s$.}
\KwOut{Connected components of $G$.}

\For{$j=1,2,\dots,N$}{
Compute the $\alpha$-lazy outward random walk $\textsf{R}_x^\alpha(t_{j-1})$ for all $x\in V$\;
Calculate the entropy $\mathfrak{D}_e^\alpha(t_{j-1})$ for all edges $e\in E$\;
Update $w_e(t_{j-1})$ to $w_e(t_j)$ for all $e\in E$ according to (\ref{weight-j})
}

\For{$\text{cutoff}=\max_{\tau\in E}w_{\tau},\cdots,\min_{\tau\in E}w_{\tau}$}{
    \For{each $e\in E$}{
        \If{$w_e > \text{cutoff}$}{
            remove edge $e$ from $E$\;
        }
    }
    Calculate community detection metrics for the post-surgery graph\;
}

\Return Best community detection metrics of $G$.
\end{algorithm}

Let us analyze the computational complexity of the entropy flow algorithm.
The computational cost of the entropy flow algorithm primarily comes from constructing the
$\alpha$-lazy outward random walks and computing the entropy for each edge. Constructing the random walk distributions for all vertices has a time complexity of $O(|V||E|)$.
In addition, computing the entropy for each edge involves summation over the entire vertex set,
which incurs an additional cost of $O(|V||E|)$. Therefore, the overall computational complexity of the entropy flow algorithm is $O(|V||E|)$.

In comparison, the computational bottleneck of the discrete Ricci curvature flow lies in solving Wasserstein distances via linear programming and computing shortest paths in the graph, with a complexity of $O(|E|D^3)$ per edge, where $D$ is the average degree of the graph.
By avoiding the solution of Wasserstein distances and linear programs, the entropy flow provides a simpler and more numerically stable implementation.

\subsection{Experimental setup}\label{experiment}
We evaluate the entropy flow algorithm on three widely used real-world networks.
The Karate network \cite{Zachary1977} is a small social network representing 34 members of a university karate club, which naturally split into two communities due to a conflict between the instructor and the club president.
The Football network represents the schedule of the 2000 NCAA Division I college football season \cite{Girvan2002community}. It consists of 115 teams as nodes, where an edge indicates that two teams played a game during the season. Due to constraints such as geography and broadcasting, teams are organized into 12 conferences.
The Facebook network is a larger and more complex social network dataset \cite{Leskovec2014SNAP}.
Each node represents a user and edges denote mutual friendships within a single ego network.
Users manually organize their friends into social circles, which serve as the ground-truth communities for this dataset.

To assess the quality of the detected communities, we adopt three standard evaluation metrics: Adjusted Rand Index (ARI), Normalized Mutual Information (NMI), and Modularity (Q) \cite{Hubert1985comparing, Danon2005, Newman2018networks}.
ARI assesses community detection accuracy by counting pairs of vertices that are
consistently assigned, while NMI measures the mutual dependence between detected
communities and ground-truth labels from an information theoretic viewpoint.
Modularity evaluates the strength of community structure based solely on network topology.
These metrics are widely used in the community detection literature and are consistent with those employed in previous studies, enabling direct comparison with existing methods.

\subsection{Effect of iterations on experimental results}\label{iterations}
We now investigate how the number of iterations influences the performance of the entropy flow algorithm.
In this set of experiments, we fix $\alpha = 0.5$ for all datasets and vary the number of iterations to examine the evolution of clustering quality.
This choice of $\alpha$ is based on extensive preliminary experiments: we observed that values in the range $0.4$--$0.6$ yield similar results across all datasets, while setting $\alpha$ too small or too large may slightly degrade performance on some networks. Therefore, $\alpha = 0.5$ represents a reasonable compromise for consistent evaluation.

For the Karate network, we set the step size to $s = 0.1$, while varying the number of iterations from 0 to 50 (Figure~\ref{fig8}(a)).
At iteration 0, the clustering performance is relatively poor (ARI = 0.38, NMI = 0.49, Modularity = 0.67), reflecting the limited community separability of the original graph.
As the number of iterations increases, all three metrics improve rapidly in the early stage, with a notable gain observed within the first 5 iterations.
Further increasing the iterations leads to a clear improvement around 20 iterations, where ARI and NMI rise to approximately 0.77 and 0.68, respectively.
Beyond this point, the metrics remain stable up to 50 iterations.

For the Football network, we set the step size to $s = 0.01$, while varying the number of iterations from 0 to 50 (Figure~\ref{fig8}(b)).
Even without applying entropy flow (iteration 0), the clustering performance is already high, with ARI = 0.89 and NMI = 0.93, reflecting the clear community structure of the network.
As the number of iterations increases, ARI and NMI exhibit a slight improvement and reach their peak within the first 5-10 iterations.
Afterwards, all three metrics remain highly stable, with ARI around 0.91, NMI around 0.93, and Modularity around 0.90 up to 50 iterations.
These results indicate that the entropy flow method is robust to the choice of iteration number.

For the Facebook network, we set the step size to $s = 0.01$, while varying the number of iterations from 0 to 50 (Figure~\ref{fig8}(c)).
The initial clustering performance is moderate (ARI = 0.69, NMI = 0.71, Modularity = 0.93). Within the first 5 iterations, all metrics improve slightly (ARI = 0.69, NMI = 0.72, Modularity = 0.96), indicating initial refinement of the community structure. From 10 to 30 iterations, the metrics fluctuate minimally, suggesting stable performance. Beyond 35 iterations, ARI and NMI gradually increase, reaching their peak at iteration 50 (ARI = 0.72, NMI = 0.73, Modularity = 0.93). Overall, these results show that the entropy flow steadily enhances community structure in the Facebook network, with slight long-term improvements over extended iterations.

\begin{figure}[H]
    \centering
    \begin{subfigure}{0.49\textwidth}
        \centering
        \includegraphics[width=\textwidth]{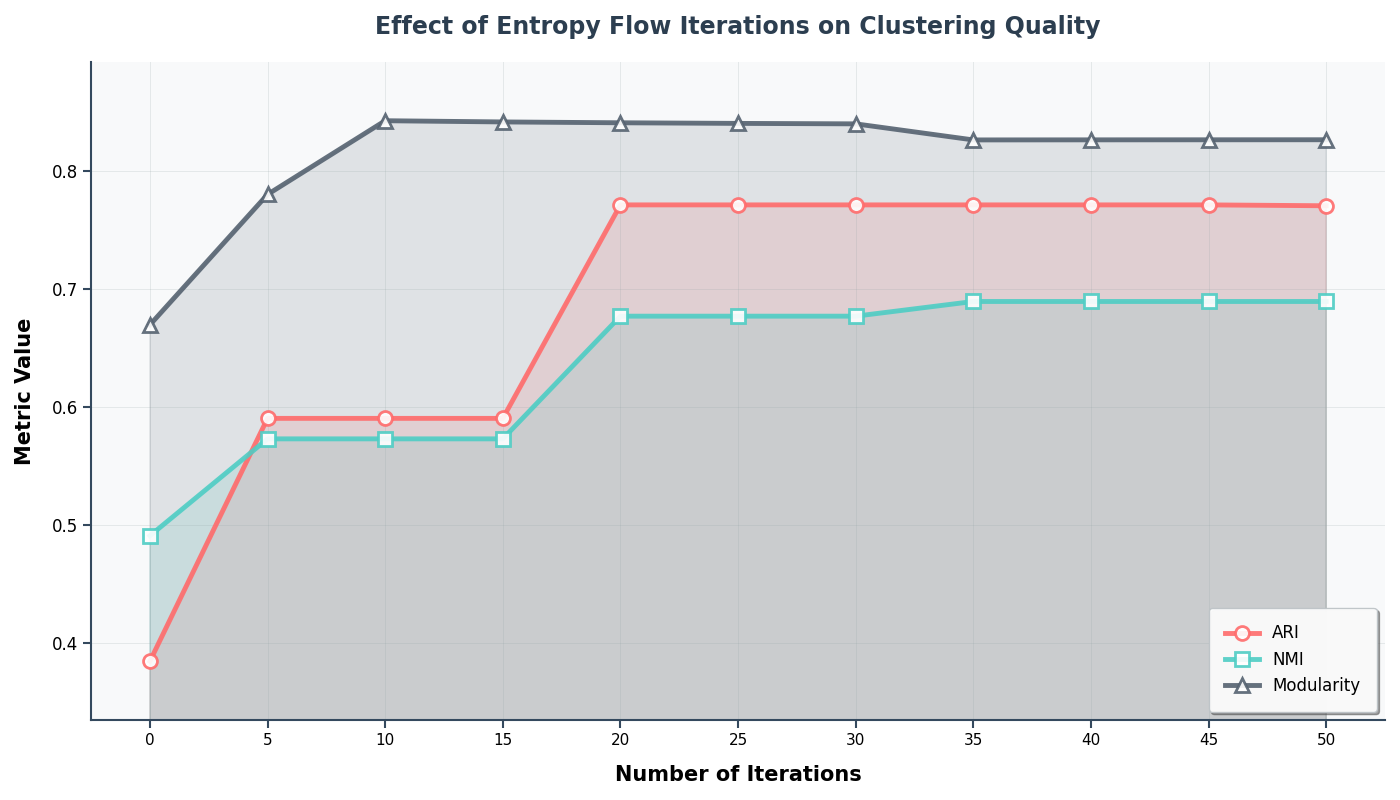}
        \caption{Karate}
    \end{subfigure}
    \hfill
    \begin{subfigure}{0.49\textwidth}
        \centering
        \includegraphics[width=\textwidth]{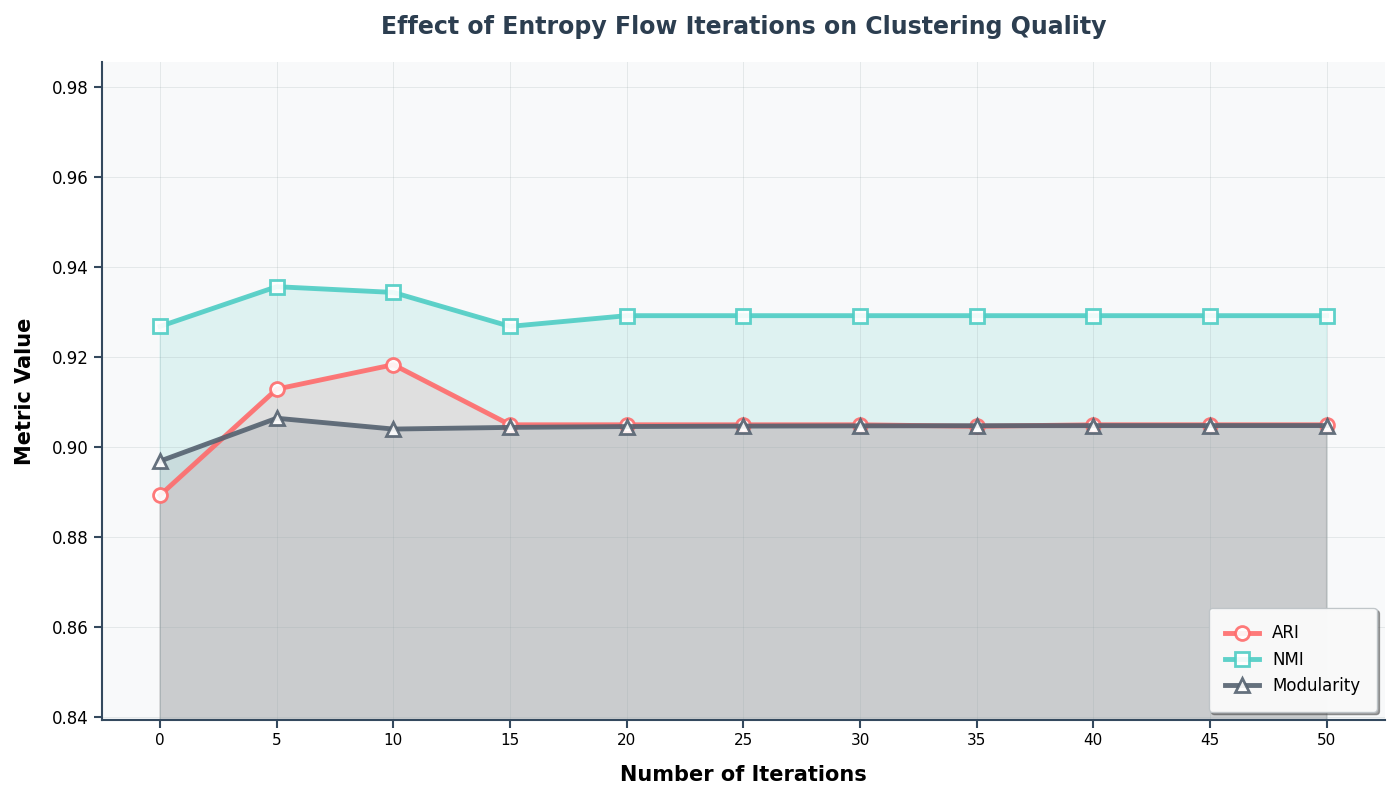}
        \caption{Football}
    \end{subfigure}
    \hfill
    \begin{subfigure}{0.49\textwidth}
        \centering
        \includegraphics[width=\textwidth]{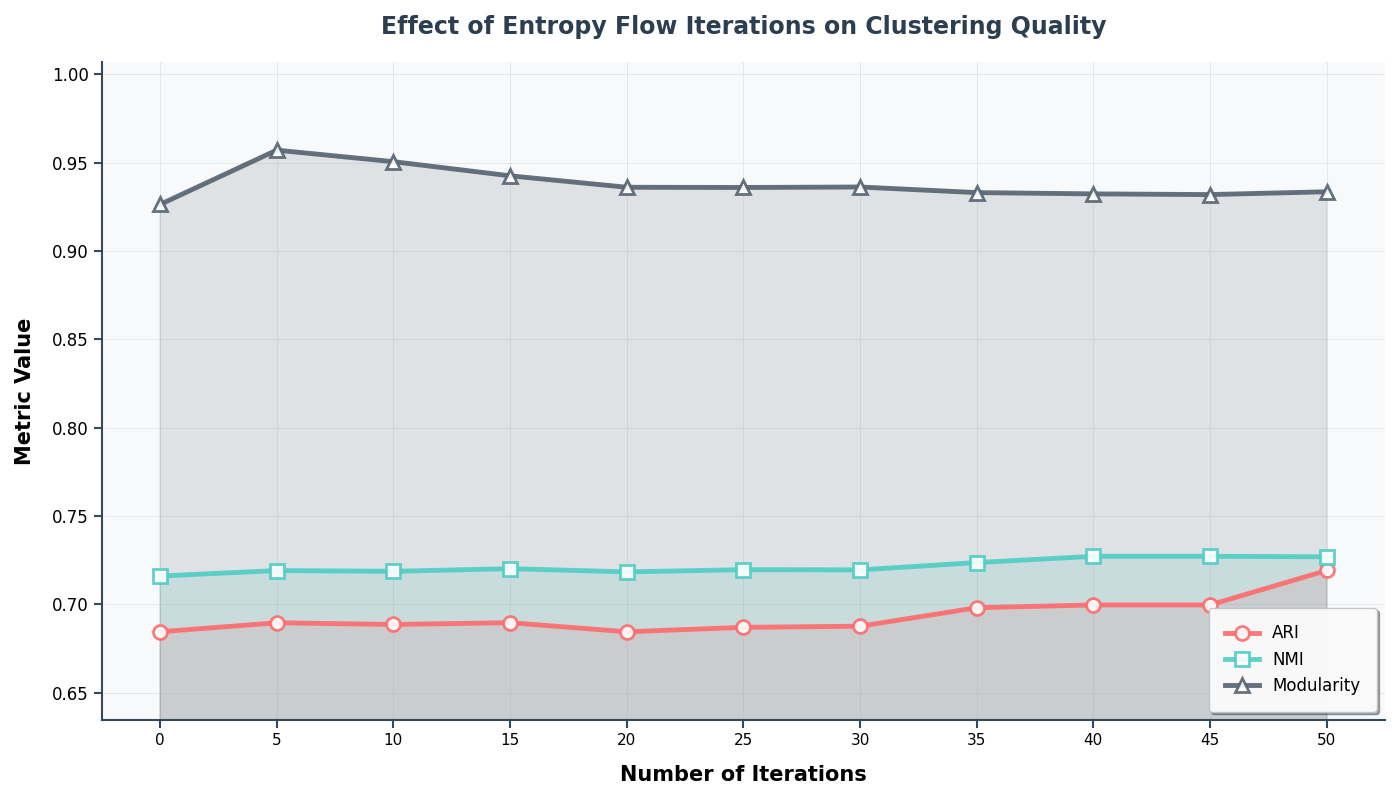}
        \caption{Facebook}
    \end{subfigure}

    \caption{Effect of iterations on three real-world datasets}
    \label{fig8}
\end{figure}

\subsection{Analysis of edge entropy and weight distributions}
\label{entropydistribution}
We investigate how the entropy flow algorithm affects the distributions of edge entropy and edge weights.
Figure~\ref{entropy-weight} illustrates the changes in edge entropy and edge weight distributions in three networks before and after  the entropy flow algorithm.
For all three networks, the edge entropy distributions are initially broad, spanning a wide range of values. In the Karate network, the distribution becomes concentrated at lower values after the flow. In the Football network, it shifts mainly to low values, while in the Facebook network, the distribution slightly shifts toward lower values and becomes more concentrated.
Edge weight distributions are initially dominated by low weight edges. In the Karate network, most edges have small weights with a few higher weight edges; after the flow, the distribution shifts toward higher weights and slightly broadens. In the Football network, edges initially concentrate on the lower range, and some move to higher weights after the flow, with a slightly wider distribution. In the Facebook network, most edges initially have low weights, and after the flow, the distribution becomes slightly more concentrated with a few edges increasing in weight.
Overall, the edge entropy distributions become more concentrated and the edge weight distributions shift toward higher or more concentrated values. The entropy flow sharpens edge weight differentiation, resulting in networks with more pronounced structure and clearer community organization.

\begin{figure}[H]
    \centering
    \begin{subfigure}{0.87\textwidth}
        \centering
        \includegraphics[width=\textwidth]{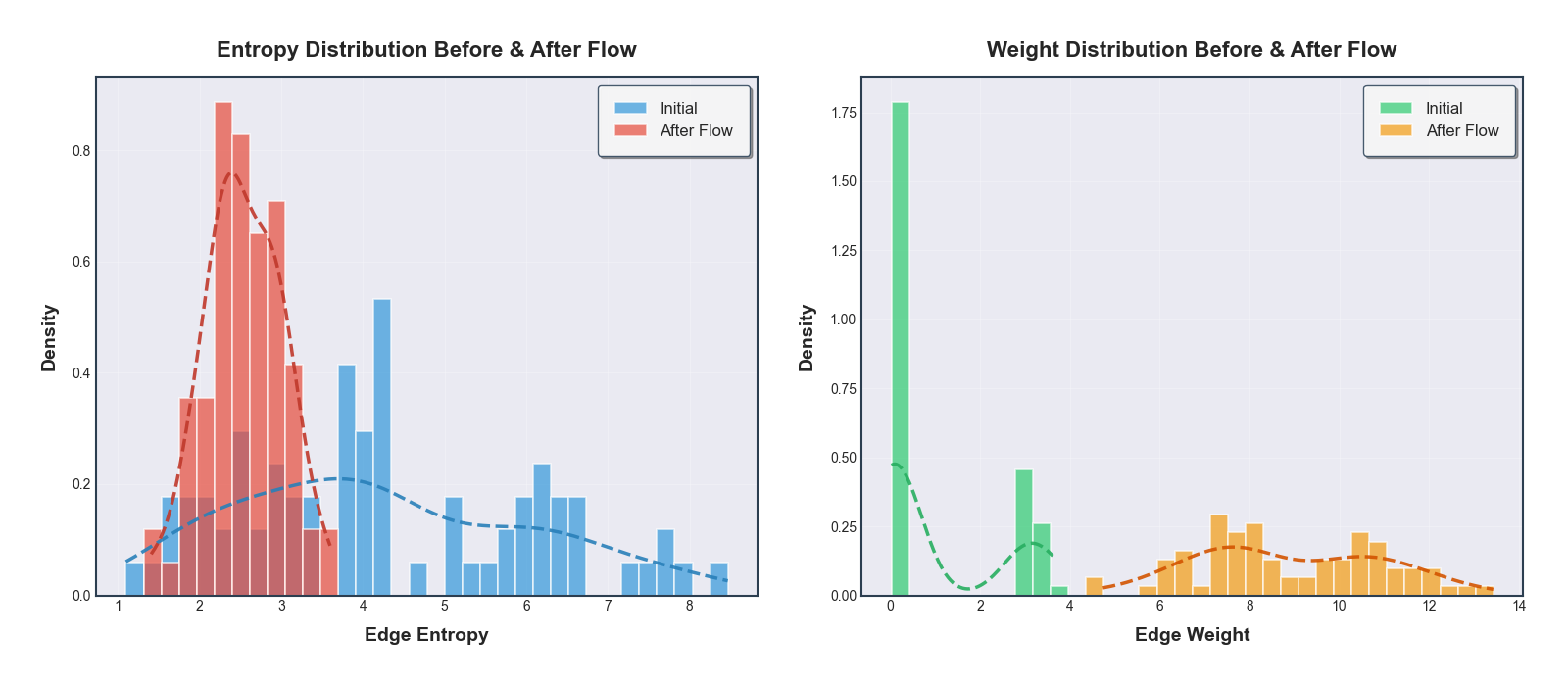}
        \caption{Karate}
        \label{fig10a}
    \end{subfigure}
    \hfill
    \begin{subfigure}{0.87\textwidth}
        \centering
        \includegraphics[width=\textwidth]{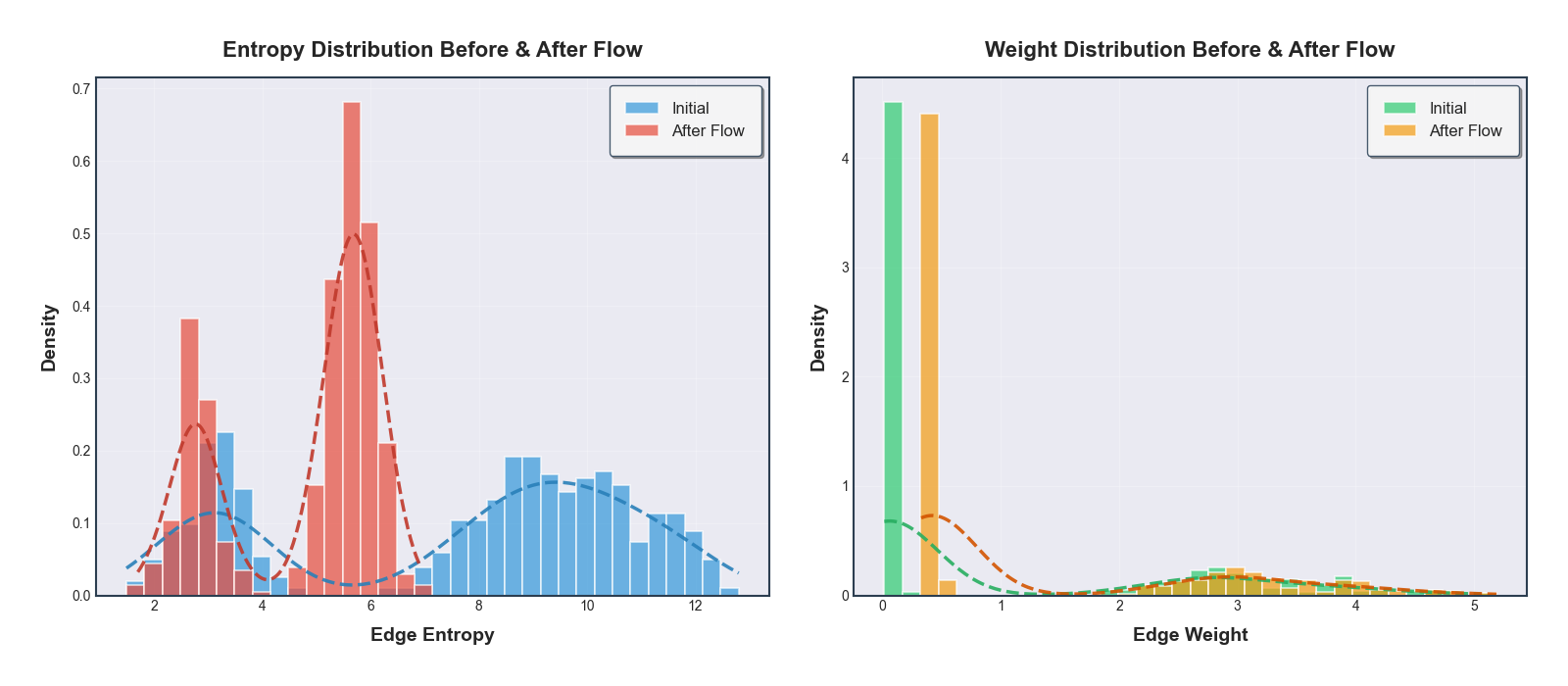}
        \caption{Football}
        \label{fig10b}
    \end{subfigure}
    \hfill
    \begin{subfigure}{0.87\textwidth}
        \centering
        \includegraphics[width=\textwidth]{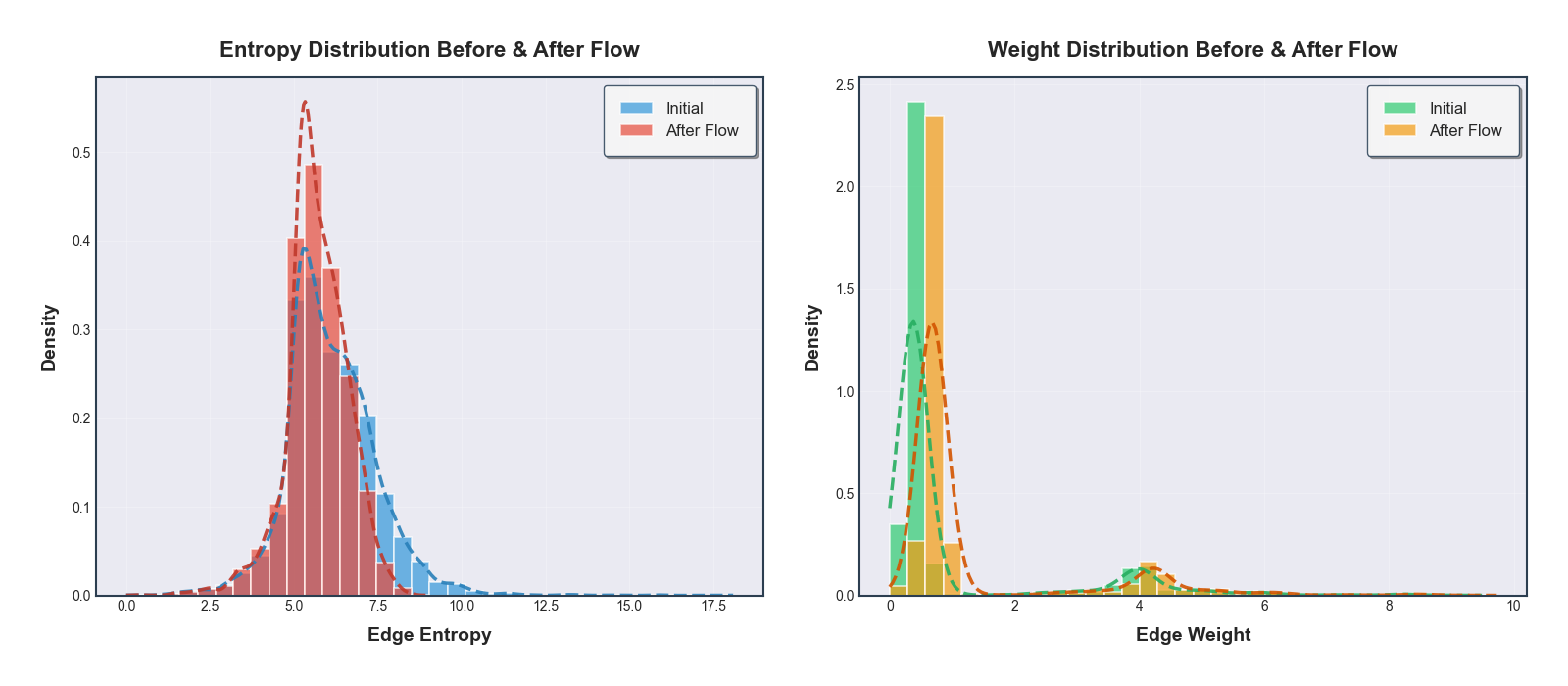}
        \caption{Facebook}
        \label{fig10c}
    \end{subfigure}
    \caption{Distribution changes of entropy values and edge weights}
    \label{entropy-weight}
\end{figure}

\subsection{Effect of surgery thresholds on experimental results}\label{threshold}
Now, we examine how varying the surgery threshold (minimal edge weight deleted) affects community structure through entropy flow. As shown in Figure~\ref{fig9}, across the Karate, Football, and Facebook networks, the metrics exhibit similar dynamic patterns with varying
surgery thresholds. At high surgery thresholds, metric values are low, indicating that removing only a few edges does not yet reveal the community structure. As the threshold decreases, Modularity rises sharply and stabilizes, reflecting stronger internal connectivity and weaker external links. ARI and NMI also increase, showing improved alignment with the ground truth and higher clustering accuracy. When the cutoff approaches the minimum edge weight, all metrics drop toward zero, as nearly all edges are removed and meaningful community partitioning becomes impossible.

\begin{figure}[H]
    \centering
    \begin{subfigure}{0.49\textwidth}
        \centering
        \includegraphics[width=\textwidth]{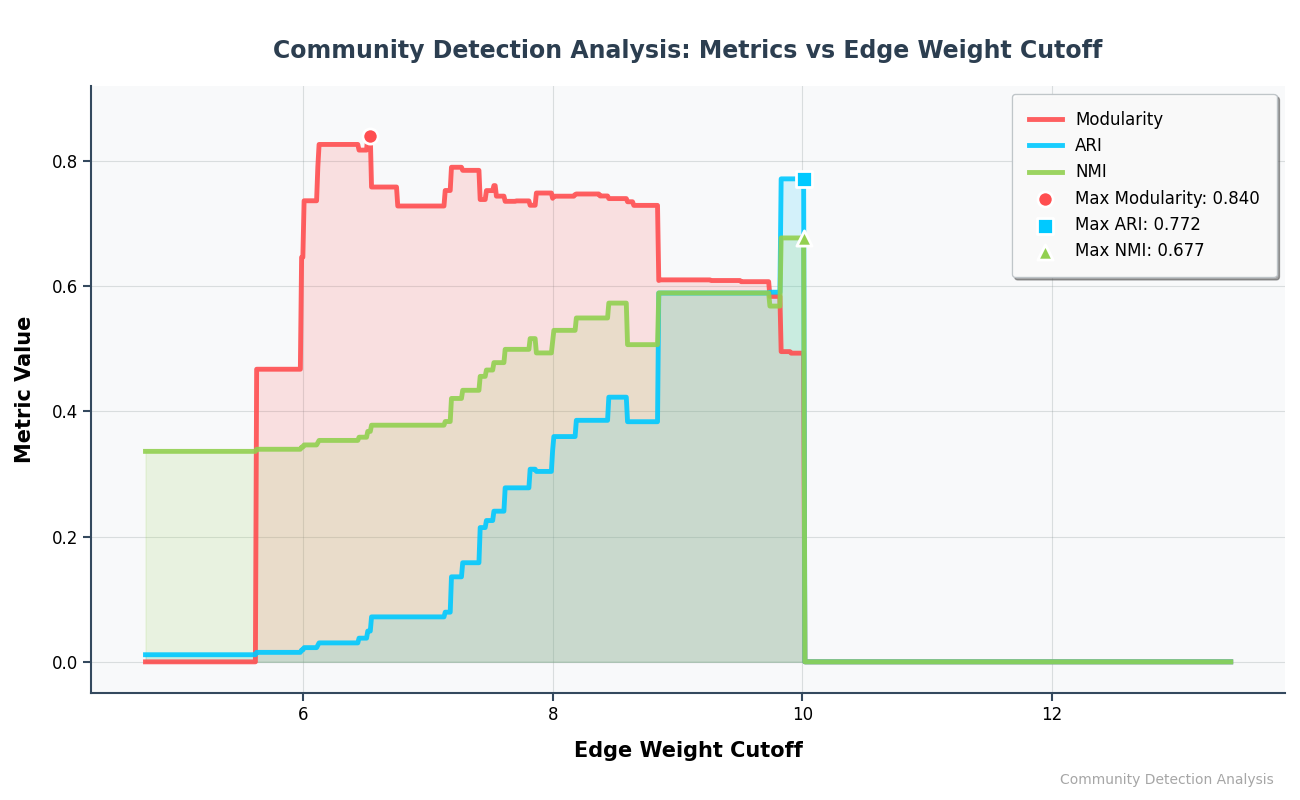}
        \caption{Karate}
    \end{subfigure}
    \hfill
    \begin{subfigure}{0.49\textwidth}
        \centering
        \includegraphics[width=\textwidth]{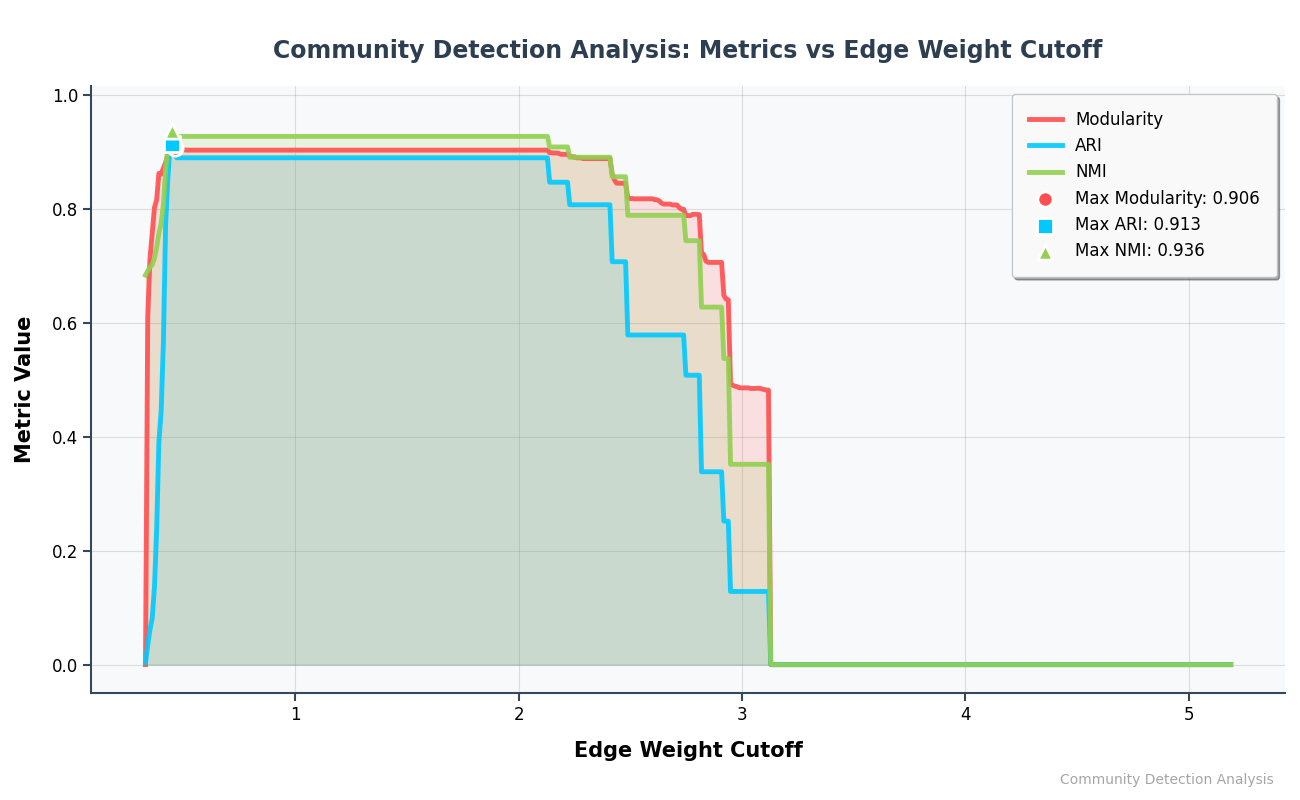}
        \caption{Football}
    \end{subfigure}
    \hfill
    \begin{subfigure}{0.49\textwidth}
        \centering
        \includegraphics[width=\textwidth]{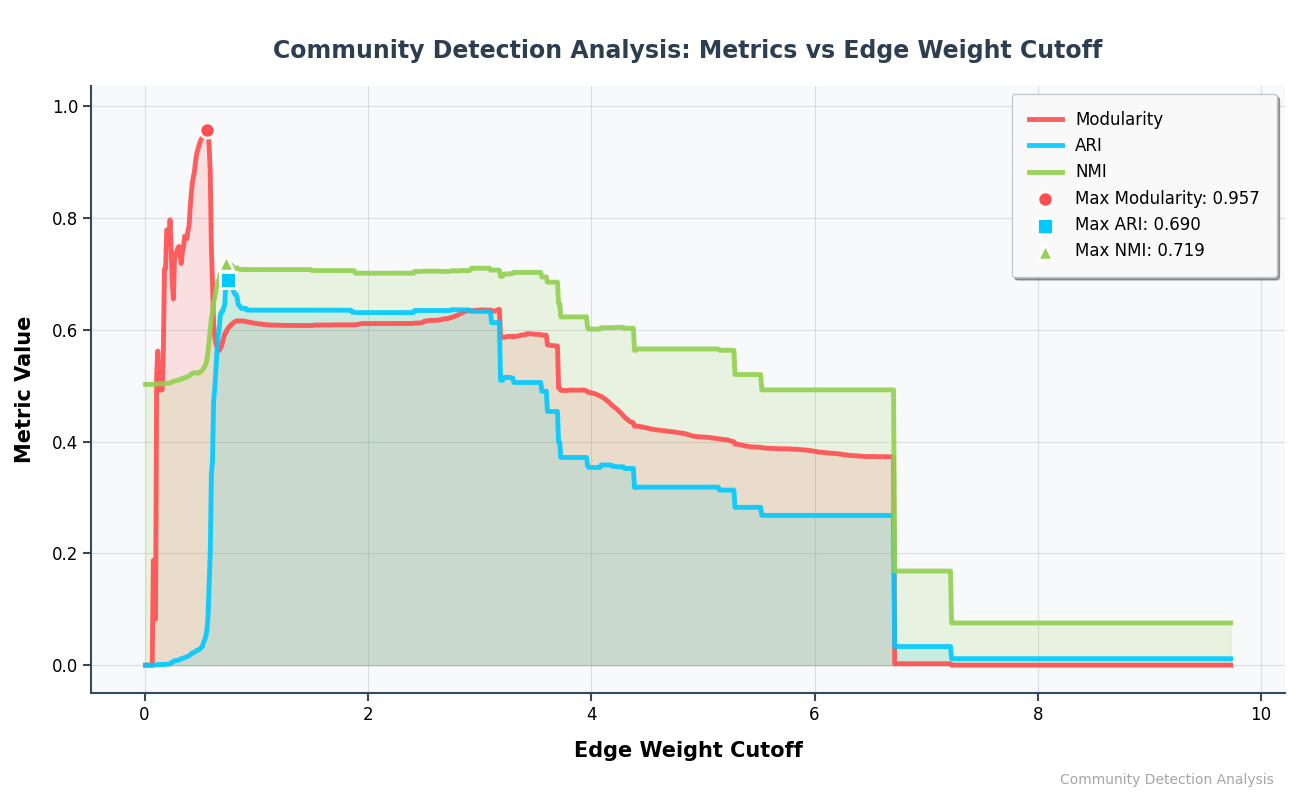}
        \caption{Facebook}
    \end{subfigure}

    \caption{Effect of edge weight cutoff on community structure}
    \label{fig9}
\end{figure}

\subsection{Comparison to other methods}\label{comparison}
We compare our community detection method with three classical approaches: the Girvan-Newman algorithm based on edge betweenness \cite{Girvan2002community}, the greedy modularity maximization algorithm \cite{Clauset2004finding, Reichardt2006statistical}, and the label propagation algorithm \cite{Cordasco2010community}. In addition, five Ricci curvature-based methods are considered, including unnormalized discrete Ollivier Ricci flow (DORF) \cite{Ni2019community}, normalized discrete Ollivier Ricci flow (NDORF), normalized discrete Lin-Lu-Yau Ricci flow (NDSRF) \cite{Lai2022normalized}, and two variants of discrete Lin-Lu-Yau Ricci flow (Rho and RhoN) \cite{Ma2025modified}.

Our method is based on the entropy flow (\ref{weight-j}) (EF).
Table~\ref{tab:community_results} summarizes the performance of all methods on three real-world networks. For EF, the parameters are set as follows: for the Karate network, $\alpha = 0.5$, $n = 30$, and $s = 0.1$; for the Football and Facebook networks, $\alpha = 0.5$, $n = 5$, and $s = 0.01$, based on the empirical analysis in previous sections. Overall, the entropy flow methods demonstrate competitive performance across all datasets.

\begin{table}[htbp]
\centering
\caption{Performance of various algorithms on real-world networks.
Best results are highlighted in red bold, and second-best results are shown in orange italic.}
\label{tab:community_results}
\begin{tabular}{lccc|ccc|ccc}
\toprule
\multicolumn{1}{c}{Methods \textbackslash Networks}
& \multicolumn{3}{c|}{Karate}
& \multicolumn{3}{c|}{Football}
& \multicolumn{3}{c}{Facebook} \\
\cmidrule(lr){2-4} \cmidrule(lr){5-7} \cmidrule(lr){8-10}
 & ARI & NMI & Q & ARI & NMI & Q & ARI & NMI & Q \\
\midrule
Girvan Newman
& \best{0.77} & \best{0.73} & 0.48
& 0.14 & 0.36 & 0.50
& 0.03 & 0.16 & 0.01 \\

Greedy Modularity
& 0.57 & 0.56 & 0.58
& 0.47 & 0.70 & 0.82
& 0.49 & 0.68 & 0.55 \\

Label Propagation
& 0.38 & 0.36 & 0.54
& {0.75} & {0.87} & {0.90}
& 0.39 & 0.65 & 0.51 \\

DORF
& 0.59 & 0.57 & 0.69
& \best{0.93} & \best{0.94} & \second{0.91}
& {0.67} & \best{0.73} & {0.68} \\

NDORF
& 0.59 & 0.57 & 0.69
& \best{0.93} & \best{0.94} & \second{0.91}
& \second{0.68} & \best{0.73} & {0.68} \\

NDSRF
& 0.59 & 0.57 & 0.68
& \best{0.93} & \best{0.94} & \second{0.91}
& \second{0.68} & \best{0.73} & {0.68} \\

Rho
& \best{0.77} & \second{0.68} & \second{0.82}
& {0.89} & {0.92} & {0.90}
& 0.64 & \second{0.72} & 0.63 \\

RhoN
& \best{0.77} & \second{0.68} & \best{0.84}
& {0.89} & \second{0.93} & \best{0.92}
& \best{0.69} & \second{0.72} & \second{0.95} \\

EF
& \best{0.77} & \second{0.68} & \best{0.84}
& \second{0.91} & \best{0.94} & \second{0.91}
& \best{0.69} & \second{0.72} & \best{0.96} \\
\bottomrule
\end{tabular}
\end{table}

\subsection{Comparison of cost time}\label{cost}
To evaluate the computational efficiency, we perform 20 iterations of the flow on each dataset with step size $s = 0.01$ and record the total running time. Each experiment is repeated 5 times, and the reported results are averaged over these runs. Under the same experimental settings, we compare the proposed entropy flow with the Ricci flow method (one-step envolution) reported in \cite{Ma2024evolution}. All experiments are conducted on identical datasets with the same initial edge weights, and only the flow update stage is considered for timing. The results demonstrate that, for all datasets, the entropy flow requires significantly less computation time than Ricci flow, highlighting its superior efficiency.

\begin{table}[H]
\centering
\caption{Comparison of  average computation time for 20 iterations. Time is measured in seconds.}
\label{tab:flow_time_comparison}
\begin{tabular}{lccc}
\toprule
Method & Karate  & Football & Facebook \\
\midrule
EF & 0.14 & 1.97 & 219.17 \\
one\_evol  & 6.71  & 61.51  & 13617.13 \\
\bottomrule
\end{tabular}
\end{table}

\section{Conclusion}
In this work, we proposed an entropy flow on weighted graphs based on the entropy between two $\alpha$-lazy outward random walks and established the existence of a unique global solution. Experimental results show that the proposed entropy flow achieves community detection performance comparable to Ricci flow while avoiding curvature optimization and distance computations, making it well suited for large-scale networks. Overall, entropy flow provides an efficient alternative to Ricci flow for uncovering network community structures.

\section*{Acknowledgements}
This research is partly supported by the National Natural Science Foundation of China (No. 12271039)
and the Open Project Program (No. K202303) of Key Laboratory of Mathematics and Complex Systems, Beijing Normal University.

\section*{Declarations}

\noindent
\textbf{Data availability}:
All data needed are available freely at https://github.com/12tangze12/Entropy-flow-on-weighted-graphs.

\noindent
\textbf{Conflict of interest}: The authors declared no potential conflicts of interest with respect to the research, authorship, and publication of this article.

\noindent
\textbf{Ethics approval}: The research does not involve humans and/or animals. The authors declare that there are no ethics issues to be approved or disclosed.


\bibliographystyle{elsarticle-num-names-alpha}

\bibliography{entropyflow}

@article{Anand2009,
  author = {Kartikeyan Anand and Ginestra Bianconi},
  title = {Entropy measures for networks: Toward an information theory of complex topologies},
  journal = {Phys. Rev. E},
  volume = {80},
  number = {4},
  pages = {045102},
  year = {2009}
}

@article{Bai2025ricci,
	title={On the {R}icci flow on trees},
	author={Bai, Shuliang and Hua, Bobo and Lin, Yong and Liu Shuang},
	journal={arXiv: 2509.22140.},
	year={2025}
}

@article{Bai2024ollivier,
	title={Ollivier {R}icci-flow on weighted graphs},
	author={Bai, Shuliang and Lin, Yong and Lu, Linyuan and Wang, Zhiyu and Yau, Shing-Tung},
	journal={Am. J. Math.},
	FJOURNAL = {American Journal of Mathematics},
	volume={146},
	number={6},
	pages={1723--1747},
	year={2024}
}

@article{Clauset2004finding,
  title = {Finding community structure in very large networks},
  author = {Clauset, Aaron and Newman, M. E. J. and Moore, Cristopher},
  journal = {Phys. Rev. E},
  volume = {70},
  issue = {6},
  pages = {066111},
  year = {2004}
}

@article{Cordasco2010community,
  title = {Community detection via semi-synchronous label propagation algorithms},
  author = {Cordasco, Gennaro and Luisa, Gargano},
  journal = {IEEE International Workshop on: Business Applications of Social Network Analysis (BASNA)},
  pages = {1-8},
  year = {2010}
}

@article{Danon2005,
author = {Danon, Leon and Díaz-Guilera, Albert and Duch, Jordi and Arenas, Alex},
title = {Comparing community structure identification},
journal = {J. Stat. Mech. Theory Exp.},
volume = {2005},
number = {09},
pages = {P09008},
year = {2005}
}

@article{Dehmer2008,
  author = {Matthias Dehmer},
  title = {Information processing in complex networks: Graph entropy and information functionals},
  journal = {Appl. Math. Comput.},
  volume = {201},
  pages = {82-94},
  year = {2012}
}

@article{Girvan2002community,
  title={Community structure in social and biological networks},
  author={Girvan, Michelle and Newman, Mark EJ},
  journal={Proc. Natl. Acad. Sci. USA},
  volume={99},
  number={12},
  pages={7821-7826},
  year={2002},
  publisher={The National Academy of Sciences}
}

@article{Hamilton1982three,
	title={Three-manifolds with positive {R}icci curvature},
	author={Hamilton, Richard S},
	journal={J. Differ. Geom.},
	FJOURNAL = {Journal of Differential geometry},
	volume={17},
	number={2},
	pages={255--306},
	year={1982},
	publisher={Lehigh University}
}

@article{Hubert1985comparing,
	title={Comparing partitions},
	author={Hubert, Lawrence and Arabie, Phipps},
	journal={J. Classification},
	volume={2},
	pages={193--218},
	year={1985}
}

@article{Kullback1951,
  author = {Solomon Kullback and Richard A. Leibler},
  title = {On information and sufficiency},
  journal = {Ann. Math. Statist.},
  volume = {22},
  pages = {79-86},
  year = {1951}
}

@article{Lai2022normalized,
	title={Normalized discrete {R}icci flow used in community detection},
	author={Lai, Xin and Bai, Shuliang and Lin, Yong},
	journal={Phys. A Stat. Mech. Appl.},
	FJOURNAL = {Physica A: Statistical Mechanics and its Applications},
	volume={597},
	pages={127251},
	year={2022},
	publisher={Elsevier}
}

@article{Leskovec2014SNAP,
title = {{SNAP} datasets: {S}tanford large network dataset collection},
author = {Leskovec, Jure},
journal = {http://snap.stanford.edu/data},
year = {2014}
}

@article{Li2026convergence,
title = {The convergence and uniqueness of a discrete-time nonlinear {M}arkov chain},
author = {Ruowei Li and Florentin Münch},
journal = {J. Funct. Anal.},
volume = {290},
number = {9},
pages = {111367},
year = {2026}
}

@article{Lin2011Ricci,
	title={Ricci curvature of graphs},
	author={Lin, Yong and Lu, Linyuan and Yau, Shing-Tung},
	journal={Tohoku Math. J.},
	volume={63},
	number={4},
	pages={605-627},
	year={2011}
}

@article{Ma2025modified,
	title={A modified {R}icci flow on arbitrary weighted graph},
	author={Ma, Jicheng and Yang, Yunyan},
	journal={J. Geom. Anal.},
    volume={35},
	pages={332},
	year={2025}
}

@article{Ma2024evolution,
	title={Evolution of weights on a connected finite graph},
	author={Ma, Jicheng and Yang, Yunyan},
	journal={arXiv: 2411.06393},
	year={2024}
}

@article{Ma2025piecewise,
	title={Piecewise-linear {R}icci curvature flows on weighted graphs},
	author={Ma, Jicheng and Yang, Yunyan},
	journal={arXiv: 2505.15395},
	year={2025}
}

@book{Newman2018networks,
  title={Networks},
  author={Newman, Mark},
  year={2018},
  publisher={Oxford Univ. Press}
}

@incollection {Ni2018network,
    AUTHOR = {Ni, Chien-Chun and Lin, Yu-Yao and Gao, Jie and Gu, Xianfeng},
     TITLE = {Network alignment by discrete {O}llivier-{R}icci flow},
 BOOKTITLE = {Graph drawing and network visualization},
    SERIES = {Lecture Notes in Comput. Sci.},
    VOLUME = {11282},
     PAGES = {447--462},
 PUBLISHER = {Springer, Cham},
      YEAR = {2018}
}

@article{Ni2019community,
	title={Community detection on networks with {R}icci flow},
	author={Ni, Chien-Chun and Lin, Yuyao and Luo, Feng and Gao, Jie},
	journal={Sci. Rep.},
	FJOURNAL = {Scientific Reports},
	volume={9},
	number={1},
	pages={9984},
	year={2019},
	publisher={Nature Publishing Group UK London}
}

@article{Ollivier2009ricci,
	title={Ricci curvature of {M}arkov chains on metric spaces},
	author={Ollivier, Yann},
	journal={J. Funct. Anal.},
	FJOURNAL = {Journal of Functional Analysis},
	volume={256},
	number={3},
	pages={810--864},
	year={2009},
	publisher={Elsevier}
}

@article{Reichardt2006statistical,
  title = {Statistical mechanics of community detection},
  author = {Reichardt, J\"org and Bornholdt, Stefan},
  journal = {Phys. Rev. E},
  volume = {74},
  issue = {1},
  pages = {016110},
  year = {2006}
}

@article{Sanhu2015graph,
  title={Graph Curvature for Differentiating Cancer Networks},
  author={Sandhu, Romeil and Georgiou, Tryphon and Reznik, Ed and Zhu, Liangjia and Kolesov, Ivan and Senbabaoglu, Yasin and Tannenbaum Allen},
  journal={Sci. Rep.},
  volume={5},
  pages={12323},
  year={2015}
}

@article{Shannon1948,
  author = {Claude E. Shannon},
  title = {A mathematical theory of communication},
  journal = {Bell System Tech. J.},
  volume = {27},
  pages = {379--423, 623--656},
  year = {1948}
}

@book{Thomas2005,
  author = {Thomas M. Cover and Joy A. Thomas},
  title = {Elements of Information Theory},
  publisher = {Wiley},
  edition = {2nd},
  year = {2005}
}

@article{Zachary1977,
author = {Zachary, Wayne W.},
title = {An information flow model for conflict and fission in small groups},
journal = {Journal of Anthropological Research},
volume = {33},
number = {4},
pages = {452-473},
year = {1977}
}

\end{document}